\documentclass[a4paper]{article}
\pdfoutput=1
\usepackage{amsfonts,amsmath,amssymb,amsthm}
\usepackage{graphicx}
\graphicspath{{./figures/}}
\usepackage{amsfonts}
\usepackage{amsmath,amssymb}
\usepackage{amsthm}
\usepackage{array}
\usepackage{csquotes}
\usepackage{enumitem}
\usepackage{subcaption}
\usepackage{graphicx}
\usepackage[dvipsnames]{xcolor}
\usepackage{latexsym}
\usepackage{mathtools}
\usepackage{tikz}
\usetikzlibrary{
    cd
}
\usepackage{type1cm}
\usepackage{bm}
\usepackage{multirow}
\usepackage[
    backend=biber,
    style=alphabetic,
    arxiv=abs,
    isbn=false,
    giveninits=true,
    maxalphanames=8,
    minalphanames=4,
    maxnames=8,
    minnames=4,]{biblatex}
\DeclareFieldFormat[misc]{title}{{#1}}
\usepackage[unicode,bookmarksdepth=3,bookmarksnumbered,bookmarksopen,pdfencoding=auto]{hyperref}
\hypersetup{
    colorlinks=true,
    linktocpage=true,
    setpagesize=false,
    citecolor=Cyan,
    linkcolor=Magenta,
    urlcolor=Green,
}
\usepackage[margin=30truemm]{geometry}
\usepackage{cleveref}
\theoremstyle{plain}
\newtheorem{theorem}{Theorem}[section]
\newtheorem{proposition}[theorem]{Proposition}
\newtheorem{lemma}[theorem]{Lemma}

\theoremstyle{definition}
\newtheorem{definition}[theorem]{Definition}
\newtheorem{example}[theorem]{Example}
\newtheorem{remark}[theorem]{Remark}
\newtheorem*{notation}{Notation}
\crefname{section}{Section}{Sections}
\crefname{subsection}{Section}{Sections}
\crefname{theorem}{Theorem}{Theorems}
\crefname{definition}{Definition}{Definitions}
\crefname{proposition}{Proposition}{Propositions}
\crefname{lemma}{Lemma}{Lemmas}
\crefname{corollary}{Corollary}{Corollaries}
\crefname{example}{Example}{Examples}
\crefname{proof}{Proof}{Proof}
\crefname{remark}{Remark}{Remarks}
\crefname{equation}{eq.}{eq.}
\crefname{figure}{Figure}{Figures}
\crefname{table}{Table}{Tables}

\numberwithin{equation}{section}
\newcommand{\id}{\mathrm{id}}
\DeclareMathOperator{\Rep}{Rep}

\title{Invariants of flat connections on 4-manifolds from Hopf group-algebras}
\author{Tomoro Mochida\thanks{Mathematical Institute, Tohoku University, 6-3, Aoba, Aramaki-aza, Aoba-ku, Sendai, 980-8578, Japan\\
\hspace*{1.8em}Email: \href{mailto:tomorou.mochida.r5@dc.tohoku.ac.jp}{\texttt{tomorou.mochida.r5@dc.tohoku.ac.jp}}}}
\date{}
\begin{document}

\maketitle

\begin{abstract}
    For a given group $G$, we construct an invariant of flat $G$-connections on 4-manifolds from a finite type involutory quasitriangular Hopf $G$-algebra. Hopf $G$-algebras are generalizations of Hopf algebras, equipped with gradings by $G$. In our construction, we color the dotted components of a Kirby diagram with elements of $G$ and employ the Hennings-type procedure. When $G$ is finite, we also define an invariant of 4-manifolds by summing the invariants over all flat $G$-connections.
\end{abstract}

\section{Introduction}

For a group $G$, Turaev~\cite{turaev2000homotopy} introduced the notion of a $G$-category in the context of homotopy quantum field theory and derived an invariant of flat $G$-connections on 3-manifolds (which he referred to as 3-dimensional $G$-manifolds). In the same paper, he also introduced Hopf $G$-coalgebras as objects whose categories of representations possess the structure of $G$-categories. When $G$ is abelian, this notion coincides with Ohtsuki's colored Hopf algebras \cite{ohtsuki1993colored}. Virelizier~\cite{virelizier2002hopf,virelizier2005involutory} examined the algebraic properties of Hopf $G$-coalgebras and showed that there are many similarities to the theory of ordinary Hopf algebras. Furthermore, he constructed two types of invariants of flat $G$-connections on 3-manifolds out of Hopf $G$-coalgebras: Hennings-type~\cite{virelizier2001algebras} and Kuperberg-type~\cite{virelizier2005involutory} invariants. Recently, De Renzi, Martel, and Wang~\cite{de2024hennings} introduced the notion of a Hopf $G$-bialgebra for an abelian group $G$, generalizing Hopf $G$-coalgebras, and constructed a $3$-dimensional topological quantum field theory.

In low-dimensional topology, many attempts have been made to construct invariants using $G$-categories and Hopf $G$-coalgebras. In the 4-dimensional case, for example, Beliakova and De~Renzi~\cite{beliakova2023refined} constructed an invariant of 4-dimensional $2$-handlebodies from Hopf $G$-coalgebras. Cui~\cite{cui2019four} and B{\"a}renz~\cite{barenz2023evaluating} each constructed an invariant of 4-manifolds from what they called $G$-crossed braided spherical fusion categories, which are $G$-categories with some additional structures. Also, Douglas and Reutter~\cite{douglas2018fusion} generalized Cui's construction using spherical fusion 2-categories, which are generalizations of $G$-crossed braided spherical fusion categories.

The main algebraic object in this paper is a Hopf $G$-algebra (\cref{def:hga}), which is the dual notion of a Hopf $G$-coalgebra. This was first suggested by Turaev~\cite{turaev2000homotopy} and explicitly formulated by Zunino~\cite{zunino2004double}. Roughly speaking, a Hopf $G$-algebra $\mathcal{H}$ is a family $\{H_\alpha\}_{\alpha\in G}$ of vector spaces equipped with a product $\{m_{\alpha,\beta}\colon H_\alpha\otimes H_\beta\to H_{\alpha\beta}\}_{\alpha,\beta\in G}$, a unit $\eta\colon\mathbb{C}\to H_1$, a coproduct $\{\Delta_\alpha\colon H_\alpha\to H_\alpha\otimes H_\alpha\}_{\alpha\in G}$, a counit $\{\varepsilon\colon H_\alpha\to\mathbb{C}\}_{\alpha\in G}$, and an antipode $\{S_\alpha\colon H_{\alpha}\to H_{\alpha^{-1}}\}_{\alpha\in G}$ that satisfy certain compatibility conditions. The sextuple $(H_1,m_{1,1},\eta,\Delta_1,\varepsilon_1,S_1)$ indexed by the identity element of $G$ is an ordinary Hopf algebra. When $G$ is trivial, the definition reduces to that of a Hopf algebra.

On the geometric side, given a group $G$, we consider flat connections on 4-manifolds (\cref{def:connection}), namely pairs $(M,\rho)$ consisting of a 4-manifold $M$ and a homomorphism $\rho\colon\pi_1(M)\to G$. Under the assumption that a Hopf $G$-algebra $\mathcal{H}$ is of finite type, involutory ($S_{\alpha^{-1}}S_{\alpha}=\id_{H_{\alpha}}$ for all $\alpha\in G$), and quasitriangular, we define an invariant of these connections. To the best of the author’s knowledge, Hopf $G$-algebras have not yet been used to construct invariants, making the present invariant a potentially new approach.

The construction takes the following steps: The algebraic data is a finite type involutory quasitriangular Hopf $G$-algebra $\mathcal{H}=\{H_\alpha\}_{\alpha\in G}$. For a flat $G$-connection $\rho$ on a 4-manifold $M$, we first calculate a presentation of the fundamental group of $M$ from its Kirby diagram. With this presentation and the connection $\rho$, we then color the dotted components of the Kirby diagram with elements of $G$. Using a 4-dimensional analog of the Hennings-type construction~\cite{hennings1996invariants,bobtcheva2003hkr}, we assign elements of $\mathcal{H}$ to the undotted components, contract these elements, and evaluate them by an integral of $H_1^*$. Applying an appropriate normalization factor, we obtain the invariant $I_\mathcal{H}(M,\rho)$.

One feature of our invariant is that it does not require the group $G$ to be either finite or abelian. However, the assumption that the Hopf $G$-algebra is involutory, which, in some sense, corresponds to semisimplicity (\cref{rem:involutorysemisimple}), is crucial. Also, the invariant is described in terms of the Hopf $G$-algebra, without using any representation theories or categorical concepts, which makes it easier to calculate. When the flat $G$-connection is trivial, our invariant coincides with B{\"a}renz and Barrett's generalized dichromatic invariant~\cite{barenz2018dichromatic} associated with the pivotal functor $\Rep(H_1)\to\Rep(D(H_1))$, where $\Rep(\cdot)$ denotes the category of finite dimensional representations and $D(H_1)$ is the Drinfeld double of $H_1$; hence it also recovers the Crane--Yetter invariant~\cite{crane1993categorical,crane1997state} for $\Rep(H_1)$.

The rest of the paper is organized as follows: In \cref{sec:hopfgrooupalgebra}, we give a review of Hopf algebras and Hopf group-algebras focusing on their similarities. In \cref{sec:coloredkirby}, we recall Kirby diagrams of 4-manifolds and explain how to obtain a presentation of the fundamental group from a Kirby diagram. Then we introduce $G$-colored Kirby diagrams of flat $G$-connections. Under these preparations, in \cref{sec:invariant}, we define the invariant and prove the invariance precisely. We calculate the invariants for several 4-manifolds in \cref{sec:example}.

\subsection*{Acknowledgements}
The author would like to thank Yuji Terashima for helpful discussions and advice. The author would also like to thank the anonymous referees for their valuable suggestions.

\section{Hopf group-algebras}\label[section]{sec:hopfgrooupalgebra}

In this section, we first briefly recall some fundamentals of Hopf algebras and list useful lemmas. Then we define Hopf group-algebras and derive their properties mainly from the theory of Hopf group-coalgebras. One will find that many of the notions in the theory of Hopf algebras can be extended to the setting of Hopf group-algebras. For more detailed treatments of Hopf algebras, see, for example, \cite{radford2012hopf} and of Hopf group-(co)algebras, see \cite{virelizier2002hopf,virelizier2005involutory,zhang2022new}.

Throughout the paper, all vector spaces are assumed to be over $\mathbb{C}$, and linear maps, dual spaces, tensor products, etc. are also unless otherwise stated. For two vector spaces $U$ and $V$, $\tau_{U,V}\colon U\otimes V\to V\otimes U;\;u\otimes v\mapsto v\otimes u$ denotes the flip map.

\subsection{Hopf algebras}

Let $H=(H,m,\eta,\Delta,\varepsilon,S)$ be a Hopf algebra. We use the following notation:

\begin{notation}
    We write the product and the unit as
    \begin{equation*}
        m(x\otimes y) = xy \quad(x,y\in H),\quad \eta(1) = 1_H.
    \end{equation*}
    For the coproduct, we use the Sweedler (sumless) notation
    \begin{equation*}
        \Delta(x) = x^{(1)}\otimes x^{(2)}\in H\otimes H \quad(x\in H),
    \end{equation*}
    and, by virtue of the coassociativity of the coproduct, extend it to
    \begin{equation*}
        \Delta^{(n-1)}(x)=x^{(1)}\otimes\cdots\otimes x^{(n)}\in H^{\otimes n}\quad(x\in H).
    \end{equation*}
\end{notation}

\subsubsection{Integrals}

In the theory of finite-dimensional Hopf algebras, integrals play an important role.

\begin{definition}
    Let $H$ be a Hopf algebra. A \emph{left} (resp. \emph{right}) \emph{integral} of $H$ is an element $\Lambda\in H$ that satisfies
    \begin{equation*}
        x\Lambda = \varepsilon(x)\Lambda\quad(\text{resp. } \Lambda x = \varepsilon(x)\Lambda)
    \end{equation*}
    for all $x\in H$. If $\Lambda$ is both a left and right integral, it is called a \emph{two-sided} integral.

    Dually, a \emph{left} (resp. \emph{right}) \emph{integral} of $H^*$ is an element $\lambda\in H^*$ that satisfies
    \begin{equation*}
        f\lambda = f(1_H)\lambda\quad (\text{resp. }\lambda f = f(1_H)\lambda)
    \end{equation*}
    for all $f\in H^*$ (in the dual Hopf algebra $H^*$), or equivalently
    \begin{equation*}
        (\id_H\otimes\lambda)(\Delta(x)) = \lambda(x)1_H\quad (\text{resp. }(\lambda\otimes\id_H)(\Delta(x)) = \lambda(x)1_H)
    \end{equation*}
    for all $x\in H$. If $\lambda$ is both a left and right integral, it is called a \emph{two-sided} integral.
\end{definition}

It is well-known that the space of left integrals and the space of right integrals of a finite-dimensional Hopf algebra are both one-dimensional~\cite[Theorem~10.2.2]{radford2012hopf}.

\begin{definition}
    A finite-dimensional Hopf algebra $H$ is said to be \emph{unimodular} if the space of left integrals and the space of right integrals of $H$ coincide.
\end{definition}

\subsubsection{Quasitriangular Hopf algebras}\label[section]{sssec:QHA}

\begin{definition}
    A \emph{quasitriangular} Hopf algebra $(H,R)$ is a pair, where $H$ is a Hopf algebra and $R\in H\otimes H$ is an invertible element, called the \emph{universal $R$-matrix}, that satisfies
    \begin{enumerate}[label=(Q\arabic*),leftmargin=*,ref=Q\arabic*]
        \item $(\Delta\otimes\id_H)(R) = R_{13}R_{23}$,\label{def:q1}
        \item $(\id_H\otimes\Delta)(R) = R_{13}R_{12}$,\label{def:q2}
        \item $R\Delta(x) = \Delta^{\text{cop}}(x)R$
    \end{enumerate}
    for all $x\in H$. Here $R_{12}\coloneq R\otimes 1_H,\;R_{23}\coloneq 1_H\otimes R,\;R_{13}\coloneq (\id_H\otimes\tau_{H,H})(R_{12})\in H^{\otimes 3}$ and $\Delta^{\text{cop}}\coloneq\tau_{H,H}\circ\Delta$.
\end{definition}

\begin{notation}
    The universal $R$-matrix is written as $R=\sum_i a_i\otimes b_i$ for some $a_i,b_i\in H$, which we abbreviate as $R=a_i\otimes b_i$.
\end{notation}

\begin{lemma}[{\cite[Theorem~12.2.8]{radford2012hopf}}]\label[lemma]{lem:proprmat}
    Let $(H,R)$ be a quasitriangular Hopf algebra. Then
    \begin{enumerate}[label=$(\mathrm{\alph*})$,leftmargin=*,ref=\alph*]
        \item $(\varepsilon\otimes\id_H)(R) = 1_H = (\id_H\otimes\varepsilon)(R)$.\label{lem:proprmata}
        \item $(S\otimes\id_H)(R) = R^{-1}$.
        \item $(S\otimes S)(R)=R$. \label{lem:proprmatc}
        \item (Quantum Yang--Baxter equation)  $R_{12}R_{13}R_{23} = R_{23}R_{13}R_{12}$.\label{lem:proprmatd}
    \end{enumerate}
\end{lemma}

For a quasitriangular Hopf algebra $(H,R)$, we define the \emph{Drinfeld element} $u\in H$ by
\begin{equation*}
    u \coloneqq (m\circ(S\otimes\id_H)\circ\tau_{H,H})(R) = S(b_i)a_i \in H.
\end{equation*}

\begin{lemma}[{\cite[Theorem~12.2.8, 12.3.2]{radford2012hopf}}]\label[lemma]{lem:drinfeldprop}
    Let $(H,R)$ be a quasitriangular Hopf algebra and $u\in H$ be its Drinfeld element. Then
    \begin{enumerate}[label=$(\mathrm{\alph*})$,leftmargin=*,ref=\alph*]
        \item $u$ is invertible with $u^{-1}=b_iS^2(a_i)$. \label{lem:drinfeldpropa}
        \item $S^2(x)=uxu^{-1}$ for all $x\in H$. \label{lem:drinfeldpropb}
        \item $uS(u)=S(u)u$ and this element belongs to the center of $H$.
        \item $\Delta(u)=(u\otimes u)(\tau_{H,H}(R)R)^{-1}=(\tau_{H,H}(R)R)^{-1}(u\otimes u)$.\label{lem:drinfeldpropd}
        \item $\varepsilon(u)=1$.\label{lem:drinfeldprope}
    \end{enumerate}
\end{lemma}

The next lemma is a direct consequence of conditions (\ref{def:q1}) and (\ref{def:q2}), and it is used several times in proving the invariance of the invariant:

\begin{lemma}\label[lemma]{lem:equatrmat}
    Let $(H,R)$ be a quasitriangular Hopf algebra and $n$ be a positive integer. Then
    \begin{enumerate}[label=$(\mathrm{\alph*})$,leftmargin=*,ref=\alph*]
	    \item $a_i^{(1)}\otimes\cdots\otimes a_i^{(n)}\otimes b_{i} = a_{i_1}\otimes\cdots\otimes a_{i_n}\otimes b_{i_1}\cdots b_{i_n}$.\label{lem:equatrmata}
        \item $a_{i}\otimes b_i^{(1)}\otimes\cdots\otimes b_i^{(n)} = a_{i_n}\cdots a_{i_1}\otimes b_{i_1}\otimes\cdots\otimes b_{i_n}$.\label{lem:equatrmatb}
    \end{enumerate}
    Here $R=a_i\otimes b_i=a_{i_1}\otimes b_{i_1}=\cdots=a_{i_n}\otimes b_{i_n}$.
\end{lemma}

\begin{proof}
    We prove part (\ref{lem:equatrmata}) by induction on $n$. For $n=1$ it is trivial, and for $n=2$ it is nothing but (\ref{def:q1}). Then for $n\geq3$, we have
    \begin{align*}
	\text{(LHS)}={}& a_i^{(1)}\otimes\cdots\otimes a_i^{(n-2)}\otimes\Delta(a_i^{(n-1)})\otimes b_i\\
	={}& a_{i_1}\otimes\cdots\otimes a_{i_{n-2}}\otimes\Delta(a_{i_{n-1}})\otimes b_{i_1}\cdots b_{i_{n-2}}b_{i_{n-1}} && \text{(assumption)}\\
	={}& a_{i_1}\otimes\cdots\otimes a_{i_{n-2}}\otimes a_{i_{n-1}}\otimes a_{i_n}\otimes b_{i_1}\cdots b_{i_{n-2}}b_{i_{n-1}}b_{i_n} && \text{(\ref{def:q1})}\\
	={}& \text{(RHS)}.
    \end{align*}
    Part (\ref{lem:equatrmatb}) is proved similarly by using (\ref{def:q2}).
\end{proof}

\begin{definition}
    A \emph{ribbon} Hopf algebra $(H,R,v)$ is a triple, where $(H,R)$ is a quasitriangular Hopf algebra and $v\in H$ is a central element, called the \emph{ribbon element}, that satisfies
    \begin{enumerate}[label=(R\arabic*),leftmargin=*,ref=R\arabic*]
        \item $v^2 = uS(u)$,\label{def:r1}
        \item $S(v) = v$,\label{def:r2}
        \item $\Delta(v) = (v\otimes v)(\tau_{H,H}(R)R)^{-1}$,\label{def:r3}
        \item $\varepsilon(v) = 1$.\label{def:r4}
    \end{enumerate}
\end{definition}

\subsubsection{Involutory Hopf algebras}\label[subsection]{sssec:invoha}

The property of being involutory leads to several powerful results.

\begin{definition}
    A Hopf algebra $H$ is said to be \emph{involutory} if the square of the antipode is the identity, i.e., $S^2=\id_H$.
\end{definition}

\begin{remark}\label[remark]{rem:involutorysemisimple}
    It is known that when a Hopf algebra $H$ is finite-dimensional, $H$ is involutory if and only if $H$ is semisimple~\cite[Theorem~16.1.2]{radford2012hopf}.
\end{remark}

\begin{lemma}[{\cite[Corollary~10.3.3]{radford2012hopf}}]
    A finite-dimensional involutory Hopf algebra is unimodular.
\end{lemma}

\begin{lemma}[{\cite[Proposition~10.4.2, Theorem~10.5.4]{radford2012hopf}}]\label[lemma]{lem:ordinvoequation}
    Let $H$ be a finite-dimensional involutory Hopf algebra. Let $\Lambda\in H$ and $\lambda\in H^*$ be (two-sided) integrals such that $\lambda(\Lambda)=1$. Then
    \begin{enumerate}[label=$(\mathrm{\alph*})$,leftmargin=*,ref=\alph*]
        \item $\dim(H) = \lambda(1)\varepsilon(\Lambda)$.\label{lem:ordinvoequationa}
        \item $\Lambda^{(1)}\otimes\Lambda^{(2)} = \Lambda^{(2)}\otimes\Lambda^{(1)}$.
        \item $S(\Lambda)=\Lambda$.\label{lem:ordinvoequationc}
        \item $\lambda(xy) = \lambda(yx)$ for all $x,y\in H$.
        \item $\lambda\circ S=\lambda$.\label{lem:ordinvoequatione}
    \end{enumerate}
\end{lemma}

\begin{lemma}\label[lemma]{lem:invoribbon}
    Let $(H,R)$ be a finite-dimensional involutory quasitriangular Hopf algebra and $u\in H$ be its Drinfeld element. Then $S(u)=u$ and $(H,R,u)$ forms a ribbon Hopf algebra.
\end{lemma}

\begin{proof}
    The first claim follows from \cite[Proposition~12.3.3]{radford2012hopf}.

    For the second claim, \cref{lem:drinfeldprop}~(\ref{lem:drinfeldpropb}) and $S^2=\id_H$ imply that $u$ is central. Conditions (\ref{def:r1}) and (\ref{def:r2}) are immediate from the first claim. Finally, (\ref{def:r3}) and (\ref{def:r4}) follow from \cref{lem:drinfeldprop}~(\ref{lem:drinfeldpropd}) and (\ref{lem:drinfeldprope}), respectively.
\end{proof}

\subsection{Hopf group-algebras}
In this section, we define Hopf group-algebras and study their properties. Let $G$ be a group with the identity element $1=1_G\in G$. 

\begin{definition}\label[definition]{def:hga}
    A \emph{Hopf $G$-algebra} $\mathcal{H}=(\{H_\alpha\}_{\alpha\in G},m,\eta,\Delta,\varepsilon,S)$ consists of  a family $\{H_\alpha\}_{\alpha\in G}$ of vector spaces and five families of linear maps
    \begin{itemize}
        \item product $m = \{m_{\alpha,\beta}\colon H_\alpha\otimes H_\beta\to H_{\alpha\beta}\}_{\alpha,\beta\in G}$,
        \item unit $\eta\colon\mathbb{C}\to H_1$,
        \item coproduct $\Delta = \{\Delta_\alpha\colon H_\alpha\to H_\alpha\otimes H_\alpha\}_{\alpha\in{G}}$,
        \item counit $\varepsilon = \{\varepsilon_\alpha\colon H_\alpha\to\mathbb{C}\}_{\alpha\in G}$,
        \item antipode $S = \{S_\alpha\colon H_\alpha\to H_{\alpha^{-1}}\}_{\alpha\in{G}}$
    \end{itemize}
    that satisfy the following conditions:
    \begin{enumerate}[label=(HG\arabic*),leftmargin=*,ref=HG\arabic*]
        \item $m_{\alpha\beta,\gamma}\circ(m_{\alpha,\beta}\otimes\id_{H_\gamma}) = m_{\alpha,\beta\gamma}\circ(\id_{H_\alpha}\otimes m_{\beta,\gamma})$,
        \item $m_{1,\alpha}\circ(\eta\otimes\id_{H_\alpha}) = \id_{H_\alpha} = m_{\alpha,1}\circ(\id_{H_\alpha}\otimes\eta)$,
        \item $(\Delta_\alpha\otimes\id_{H_\alpha})\circ \Delta_\alpha = (\id_{H_\alpha}\otimes\Delta_\alpha)\circ \Delta_\alpha$,\label{def:hga3}
        \item $(\varepsilon_\alpha\otimes\id_{H_\alpha})\circ \Delta_\alpha = \id_{H_\alpha} = (\id_{H_\alpha}\otimes\varepsilon_\alpha)\circ\Delta_\alpha$,\label{def:hga4}
        \item $\Delta_{\alpha\beta}\circ m_{\alpha,\beta} = (m_{\alpha,\beta}\otimes m_{\alpha,\beta})\circ(\id_{H_\alpha}\otimes\tau_{H_\alpha,H_\beta}\otimes\id_{H_\beta})\circ(\Delta_\alpha\otimes\Delta_\beta)$,
        \item $\varepsilon_{\alpha\beta}\circ m_{\alpha,\beta} = \varepsilon_\alpha\otimes\varepsilon_\beta$,\label{def:hga6}
        \item $\Delta_1\circ\eta = \eta\otimes\eta$,
        \item $\varepsilon_1\circ\eta = \id_{\mathbb{C}}$,\label{def:hga8}
        \item $m_{\alpha^{-1},\alpha}\circ(S_\alpha\otimes\id_{H_\alpha})\circ\Delta_\alpha = \eta\varepsilon_\alpha = m_{\alpha,\alpha^{-1}}\circ(\id_{H_\alpha}\otimes S_\alpha)\circ\Delta_\alpha$\label{def:hga9}
    \end{enumerate}
    for all $\alpha,\beta,\gamma\in G$.

    A Hopf $G$-algebra $\mathcal{H}=\{H_\alpha\}_{\alpha\in G}$ is said to be of \emph{finite type} if $H_\alpha$ is finite-dimensional for all $\alpha\in G$.
\end{definition}

\begin{remark}\label[remark]{rem:hgaremarkhga}
    If $G$ is trivial, the definition reduces to that of an ordinary Hopf algebra. In this sense, Hopf $G$-algebras can be regarded as generalizations of Hopf algebras. Also, $(H_\alpha,\Delta_{\alpha,\alpha},\varepsilon_\alpha)$ is a coalgebra for all $\alpha\in G$, and $(H_1,m_1,\eta,\Delta_{1,1},\varepsilon_1,S_1)$ is a Hopf algebra.
\end{remark}

\begin{notation}
    Similar to the case of Hopf algebras, we use the following notation for the product, the unit, and the coproduct of $\mathcal{H}$:
    \begin{equation*}
        m_{\alpha,\beta}(x\otimes y) = xy,\quad \eta(1)=1_{H_1},\quad \Delta_\alpha(x) = x^{(1)}\otimes x^{(2)} \quad(\alpha,\beta\in G,x\in H_\alpha, y\in H_\beta).
    \end{equation*}
\end{notation}

The following lemma states the basic properties of the antipode:

\begin{lemma}[{\cite[Section~2.3]{zhang2022new}}]\label[lemma]{lem:prophga}
    Let $\mathcal{H}$ be a  Hopf $G$-algebra. Then
    \begin{enumerate}[label=$(\mathrm{\alph*})$,leftmargin=*,ref=\alph*]
        \item $S_{\alpha\beta}(xy) = S_\beta(y)S_\alpha(x)$ for all $\alpha,\beta\in G$, $x\in H_\alpha$ and $y\in H_\beta$. \label{lem:prophgaa}
        \item $S_1(1_{H_1}) = 1_{H_1}$.
        \item $S_\alpha(x)_{(1)}\otimes S_\alpha(x)_{(2)} = S_\alpha(x_{(2)})\otimes S_\alpha(x_{(1)})$ for all $\alpha\in G$ and $x\in H_\alpha$. \label{lem:prophgac}
        \item $\varepsilon_{\alpha^{-1}}\circ  S_\alpha = \varepsilon_\alpha$ for all $\alpha\in G$. \label{lem:prophgad}
    \end{enumerate}
\end{lemma}

\begin{lemma}[{\cite[Corollary~4]{zhang2022new}}]\label[lemma]{lem:nonzerosubgrhga}
    Let $\mathcal{H}$ be a Hopf $G$-algebra. Then $\{\alpha\in G\mid H_\alpha\neq0\}$ is a subgroup of $G$.
\end{lemma}

\begin{remark}
    For finite type Hopf $G$-algebras, many of their properties can be derived by dualizing the corresponding notions in Hopf $G$-coalgebras. However, we suggest that these properties could also be obtained without relying on the theory of Hopf $G$-coalgebras but by developing a standalone theory of Hopf $G$-algebras. Given the requirement for a Hopf $G$-algebra to be of finite type in the construction of the invariant, we have decided to make use of dual statements, if necessary, in the following.
\end{remark}

\subsubsection{\texorpdfstring{$G$}{G}-integrals}

The notion of a $G$-integral of a Hopf $G$-algebra corresponds to the notion of an integral of a Hopf algebras.

\begin{definition}
    Let $\mathcal{H}$ be a finite type Hopf $G$-algebra. A \emph{left} (resp. \emph{right}) \emph{$G$-integral} of $\mathcal{H}$ is an element $\Lambda=(\Lambda_\alpha)_{\alpha\in G}\in\prod_{\alpha\in G}H_\alpha$ that satisfies
    \begin{equation*}
        x\Lambda_\beta = \varepsilon_\alpha(x)\Lambda_{\alpha\beta}\quad \left(\text{resp. } \Lambda_\alpha y = \varepsilon_\beta(y)\Lambda_{\alpha\beta}\right)
    \end{equation*}
    for all $\alpha,\beta\in G$, $x\in H_\alpha$ and $y\in H_\beta$. If $\Lambda$ is a both left and right $G$-integral, it is called a \emph{two-sided} $G$-integral. 
    
    A $G$-integral $\Lambda=(\Lambda_\alpha)_{\alpha\in G}$ is said to be \emph{nonzero} if $\Lambda_\alpha\neq0$ for some $\alpha\in G$.
\end{definition}

\begin{lemma}\label[lemma]{lem:evalginthga}
    Let $\mathcal{H}$ be a finite type Hopf $G$-algebra and $\Lambda=(\Lambda_\alpha)_{\alpha\in G}$ be a left (or right) $G$-integral of $\mathcal{H}$. Then $\varepsilon_\alpha(\Lambda_\alpha)=\varepsilon_\beta(\Lambda_\beta)$ for any $\alpha,\beta\in G$ with $H_\alpha\neq0$ and $H_\beta\neq0$.
\end{lemma}

\begin{proof}
    Let us consider the left $G$-integral case. The condition for $\Lambda$ to be a left $G$-integral can be rewritten as
    \begin{equation*}
        m_{\alpha,\beta}(\id_{H_\alpha}\otimes\Lambda_\beta) = \Lambda_{\alpha\beta}\varepsilon_\alpha
    \end{equation*}
    for any $\alpha,\beta\in G$. Replacing $\alpha$ with $\alpha\beta^{-1}$ and applying $\varepsilon_\alpha$ to both sides of the equation yields
    \begin{equation*}
        \varepsilon_{\beta}(\Lambda_\beta)\varepsilon_{\alpha\beta^{-1}} = \varepsilon_\alpha(\Lambda_{\alpha})\varepsilon_{\alpha\beta^{-1}}.
    \end{equation*}
    Now $H_{\alpha\beta^{-1}}\neq0$ by \cref{lem:nonzerosubgrhga}, and hence $\varepsilon_{\alpha\beta^{-1}}\neq0$. Therefore, we have $\varepsilon_\alpha(\Lambda_\alpha)=\varepsilon_\beta(\Lambda_\beta)$. The right $G$-integral case can be proved similarly.
\end{proof}

The following lemma is one of the main results of \cite{virelizier2002hopf}, which ensures the existence and uniqueness of a $G$-integral of a finite type Hopf $G$-(co)algebra:

\begin{lemma}[{Dual of \cite[Theorem~3.6]{virelizier2002hopf}}]\label[lemma]{lem:uniqueginthga}
    Let $\mathcal{H}$ be a finite type Hopf $G$-algebra. Then the space of left $G$-integrals and the space of right $G$-integrals of $\mathcal{H}$ are both one-dimensional, that is, if $\Lambda=(\Lambda_\alpha)_{\alpha\in G}$ and $\Lambda'=(\Lambda'_\alpha)_{\alpha\in G}$ are nonzero left (or right) $G$-integrals of $\mathcal{H}$, then there is some scalar $k\in\mathbb{C}$ such that $\Lambda'_\alpha=k\Lambda_\alpha$ for all $\alpha\in G$.
\end{lemma}

\begin{lemma}[{Dual of \cite[Corollary~3.7~(a)]{virelizier2002hopf}}]\label[lemma]{lem:bijanthga}
    The antipode $S=\{S_\alpha\}_{\alpha\in G}$ of a finite type Hopf $G$-algebra $\mathcal{H}$ is bijective, i.e., $S_\alpha$ is bijective for all $\alpha\in G$.
\end{lemma}

\begin{definition}
    A Hopf $G$-algebra $\mathcal{H}$ is said to be \emph{unimodular} if the space of left $G$-integrals and the space of right $G$-integrals of $\mathcal{H}$ coincide.
\end{definition}

\subsubsection{Quasitriangular Hopf-group algebras}

In order to define quasitriangular Hopf $G$-algebras, the notion of crossing is needed.

\begin{definition}\label[definition]{def:crossinghga}
    A \emph{crossed} Hopf $G$-algebra $(\mathcal{H},\varphi)$ is a pair, where $\mathcal{H}$ is a Hopf $G$-algebra and $\varphi=\{\varphi_\beta^\alpha\colon H_\alpha\to H_{\beta\alpha\beta^{-1}}\}$ is a family of linear maps, called the \emph{crossing}, that satisfies
    \begin{enumerate}[label=(CHG\arabic*),leftmargin=*,ref=CHG\arabic*]
        \item $\varphi_\beta^\alpha\colon H_\alpha\to H_{\beta\alpha\beta^{-1}}$ is a coalgebra isomorphism,\footnote{Recall that each $H_\alpha$ is a $\mathbb{C}$-coalgebra (\cref{rem:hgaremarkhga}).}
        \item $m_{\beta\alpha\beta^{-1},\beta\gamma\beta^{-1}}\circ(\varphi_\beta^\alpha\otimes\varphi_\beta^\gamma) = \varphi_\beta^{\alpha\gamma}\circ m_{\alpha,\gamma}\colon H_\alpha\otimes H_\gamma\to H_{\beta\alpha\gamma\beta^{-1}}$,
        \item $\varphi_\beta^1(1_{H_1}) = 1_{H_1}$,
        \item $\varphi_{\beta\beta'}^\alpha = \varphi_\beta^{\beta'\alpha\beta'^{-1}}\circ\varphi_{\beta'}^\alpha\colon H_\alpha\to H_{\beta\beta'\alpha\beta'^{-1}\beta^{-1}}$
    \end{enumerate}
    for all $\alpha,\beta,\beta',\gamma\in G$.
\end{definition}

\begin{remark}\label[remark]{rem:crossabelhga}
    When $G$ is abelian, $\mathcal{H}$ admits the trivial crossing by setting $\varphi_\beta^\alpha\coloneqq\id_{H_\alpha}$ for all $\alpha,\beta\in G$.
\end{remark}

\begin{lemma}[{\cite[Lemma~13]{zhang2022new}}]\label[lemma]{lem:crosshga}
    Let $(\mathcal{H},\varphi)$ be a crossed Hopf $G$-algebra. Then
    \begin{enumerate}[label=$(\mathrm{\alph*})$,leftmargin=*,ref=\alph*]
        \item $\varphi_1^\alpha=\id_{H_{\alpha}}$ for all $\alpha\in G$.
        \item $(\varphi_\beta^\alpha)^{-1}=\varphi_{\beta^{-1}}^{\beta\alpha\beta^{-1}}$ for all $\alpha,\beta\in G$.
        \item $\varphi_\beta^{\alpha^{-1}}\circ S_\alpha=S_{\beta\alpha\beta^{-1}}\circ\varphi_\beta^\alpha$ for all $\alpha,\beta\in G$.
        \item Let $\Lambda=(\Lambda_\alpha)_{\alpha\in G}$ be a left (resp. right) $G$-integral of $\mathcal{H}$. Then $(\varphi_\beta^\alpha(\Lambda_\alpha))_{\alpha\in G}$ is also a left (resp. right) $G$-integral of $\mathcal{H}$ for all $\beta\in G$,.
    \end{enumerate}
\end{lemma}

We are now ready to define a quasitriangular Hopf $G$-algebra.

\begin{definition}\label[definition]{def:qhga}
    A \emph{quasitriangular} Hopf $G$-algebra $(\mathcal{H},\varphi,R)$ is a triple, where $(\mathcal{H},\varphi)$ is a crossed Hopf $G$-algebra and $R\in H_1\otimes H_1$ is an invertible element, called the \emph{universal $R$-matrix}, that satisfies
    \begin{enumerate}[label=(QHG\arabic*),leftmargin=*,ref=QHG\arabic*]
        \item $(\Delta_1\otimes\id_{H_1})(R) = R_{13}R_{23}$,\label{lem:qhga1}
        \item $(\id_{H_1}\otimes\Delta_1)(R) = R_{13}R_{12}$,\label{lem:qhga2}
        \item $R\Delta_\alpha(x) = \Delta_\alpha^{\text{cop}}(x)R$,\label{lem:qhga3}
        \item $(\varphi_\beta^1\otimes\varphi_\beta^1)(R)=(R)$\label{lem:qhga4}
    \end{enumerate}
    for all $\alpha,\beta\in G$ and $x\in H_\alpha$. Here $R_{12}\coloneq R\otimes 1_{H_1},\;R_{23}\coloneq 1_{H_1}\otimes R,\;R_{13}\coloneq (\id_{H_1}\otimes\tau_{H_1,H_1})(R_{12})\in H_1^{\otimes 3}$ and $\Delta_\alpha^{\text{cop}}\coloneq\tau_{H_\alpha,H_\alpha}\circ\Delta$.
\end{definition}

\begin{remark}
    In this case, $(H_1,R)$ is a quasitriangular Hopf algebra. Therefore, all the lemmas discussed in \cref{sssec:QHA} hold in the context of quasitriangular Hopf $G$-algebras.
\end{remark}

\begin{remark}
    As far as the author knows, there have been two different definitions of quasitriangular Hopf $G$-algebras: one given in \cite{zhao2014quasitriangular,ma2015class}, and the other in \cite{zhang2022new}. Our definition falls between the two, being stronger than the former and weaker than the latter. Specifically, in the former, condition~(\ref{lem:qhga4}) is not imposed  (and therefore, the crossing structure is not necessary). In the latter, the universal $R$-matrix is defined to be a family $\{R_{\alpha,\beta}\in H_\alpha\otimes H_\beta\}_{\alpha,\beta\in G}$ that satisfies certain conditions, rather than a single element $R\in H_1\otimes H_1$ as we do. Restricting to conditions that involve only $R_{1,1}\in H_1\otimes H_1$ allows us to recover our definition.
\end{remark}

\begin{notation}
    As in the case of quasitriangular Hopf algebras, we use the same notation $R=a_i\otimes b_i$.
\end{notation}

\subsubsection{Involutory Hopf group-algebras}

Similar to involutory Hopf algebras, involutory Hopf $G$-algebras also behave very well.

\begin{definition}
    A Hopf $G$-algebra $\mathcal{H}$ is said to be \emph{involutory} if the antipode satisfies $S_{\alpha^{-1}}\circ S_{\alpha}=\id_{H_\alpha}$ for all $\alpha\in G$.
\end{definition}

\begin{lemma}[{Dual of \cites[Lemma~3]{virelizier2005involutory}[Corollary~5.7]{virelizier2002hopf}}]\label[lemma]{lem:invounihga}
    Let $\mathcal{H}$ be a finite type involutory Hopf $G$-algebra such that $\dim(H_1)\neq0$. Then $\mathcal{H}$ is unimodular.
\end{lemma}

\begin{lemma}\label[lemma]{lem:intantihga}
    Let $\mathcal{H}$ be a finite type Hopf $G$-algebra and $\Lambda=(\Lambda_\alpha)_{\alpha\in G}$ be a left (resp. right) $G$-integral of $\mathcal{H}$. Then $S_G(\Lambda)\coloneqq(S_\alpha(\Lambda_\alpha))_{\alpha\in G}$ is a right (resp. left) integral of $\mathcal{H}$. Moreover, if $\mathcal{H}$ is involutory, then $S_\alpha(\Lambda_\alpha)=\Lambda_{\alpha^{-1}}$ for all $\alpha\in G$.
\end{lemma}

\begin{proof}
    Let $\Lambda=(\Lambda_\alpha)_{\alpha\in G}$ be a nonzero left $G$-integral of $\mathcal{H}$. For any $\alpha,\beta\in G$ and $x\in H_\beta$, noting that $S_{\alpha^{-1}}(\Lambda_{\alpha^{-1}})\in H_\alpha$, the calculation
    \begin{align*}
        S_{\alpha^{-1}}(\Lambda_{\alpha^{-1}})x &= S_{\alpha^{-1}}(\Lambda_{\alpha^{-1}})S_{\beta^{-1}}(S_{\beta^{-1}}^{-1}(x)) &&\text{(\Cref{lem:bijanthga})}\\
        &= S_{\beta^{-1}\alpha^{-1}}(S_{\beta^{-1}}^{-1}(x)\Lambda_{\alpha^{-1}}) && \text{(\Cref{lem:prophga}~(\ref{lem:prophgaa}))}\\
        &= S_{(\alpha\beta)^{-1}}(\varepsilon_{\beta^{-1}}(S_{\beta^{-1}}^{-1}(x))\Lambda_{(\alpha\beta)^{-1}}) && \text{($\Lambda$ is a $G$-integral)}\\
        &= \varepsilon_\beta(x)S_{(\alpha\beta)^{-1}}(\Lambda_{(\alpha\beta)^{-1}}) && \text{(\Cref{lem:prophga}~(\ref{lem:prophgad}))}
    \end{align*}
    shows that $S_G(\Lambda)$ is a right $G$-integral of $\mathcal{H}$. The right $G$-integral case can be proved similarly.

    If $\mathcal{H}$ is involutory, then $\mathcal{H}$ is unimodular by \cref{lem:invounihga}. Since the space of $G$-integrals is one-dimensional (\cref{lem:uniqueginthga}), there is $k\in\mathbb{C}$ such that $S_\alpha(\Lambda_\alpha)=k\Lambda_{\alpha^{-1}}$ for all $\alpha\in G$. Applying $\varepsilon_{\alpha^{-1}}$ to both sides yields $k=1$.
\end{proof}

\begin{lemma}[{Dual of \cite[Corollary~4.4]{virelizier2002hopf}}]\label[lemma]{lem:cyccointhga}
    Let $\mathcal{H}$ be a finite type involutory Hopf $G$-algebra and $\lambda$ be a (two-sided) integral of $H_1^*$. Then $\lambda(xy) = \lambda(yx)$ for any $\alpha\in G$, $x\in H_\alpha$ and $y\in H_{\alpha^{-1}}$.
\end{lemma}

\begin{lemma}[{Dual of \cite[Lemma~7.2, 6.3, Corollary~6.2]{virelizier2002hopf}}]\label[lemma]{lem:crossinthga}
    Let $(\mathcal{H},\varphi)$ be a finite type involutory crossed Hopf $G$-algebra. Let $\Lambda=(\Lambda_\alpha)_{\alpha\in G}$ be a (two-sided) $G$-integral of $\mathcal{H}$ and $\lambda$ be a (two-sided) integral of $H_1^*$. Then
    \begin{enumerate}[label=$(\mathrm{\alph*})$,leftmargin=*,ref=\alph*]
        \item $\varphi_\beta^\alpha(\Lambda_\alpha)=\Lambda_{\beta\alpha\beta^{-1}}$ for all $\alpha,\beta\in G$.\label{lem:crossinthgaa}
        \item $\lambda\circ\varphi_\beta^1=\lambda$ for all $\beta\in G$.\label{lem:crossinthgab}
    \end{enumerate}
\end{lemma}

We end this section with two examples of involutory quasitriangular Hopf $G$-algebras.

\begin{example}\label[example]{ex:cyclicgrhga}
    Let $\mathbb{Z}_m=\langle g\mid g^m\rangle$ be a cyclic group of order $m$. Choose an arbitrary factorization of $m$ into two positive integers $k,l$: $m=kl$. We set $\mathbb{Z}_k=\langle\alpha\mid\alpha^k\rangle$ and construct a Hopf $\mathbb{Z}_k$-algebra $\mathcal{H}=\{H_{\alpha^p}\}_{p=0,\ldots,k-1}$ as follows:

    Let $\omega=e^\frac{2\pi\sqrt{-1}}{l}$ be an $l$-th root of unity. For each $p=0,\ldots,k-1$, $H_{\alpha^p}$ is a vector space with a basis $\{g^p,g^{k+p},\allowbreak \ldots,g^{(l-1)k+p}\}$. We then define structure morphisms on $\mathcal{H}$ by
    \begin{description}[labelwidth=5em,labelindent=2em,font=\normalfont]
        \item[product] $m_{\alpha^p,\alpha^q}(g^{ik+p}\otimes g^{jk+q}) \coloneqq g^{(i+j)k+p+q}$,
        \item[unit] $\eta(1) \coloneqq 1$,
        \item[coproduct] $\Delta_{\alpha^p}(g^{ik+p}) \coloneqq g^{ik+p}\otimes g^{ik+p}$,
        \item[counit] $\varepsilon_{\alpha^p}(g^{ik+p}) \coloneqq 1$,
        \item[antipode] $S_{\alpha^p}(g^{ik+p}) \coloneqq g^{-ik-p}$.
    \end{description}
    Here $p,q=0,\ldots,k-1$ and $i,j=0,\ldots,l-1$, and exponents are taken modulo $m$.\footnote{This structure is derived from the Hopf algebra structure of the group ring $\mathbb{C}[\mathbb{Z}_m]$.} Since $\mathbb{Z}_k$ is abelian, $\mathcal{H}$ admits the trivial crossing (\cref{rem:crossabelhga}) and also has universal $R$-matrices given by
    \begin{equation*}
        R_d = \frac{1}{l}\sum_{i,j=0}^{l-1}\omega^{-ij}g^{ik}\otimes g^{djk} \in H_1\otimes H_1,\quad d\in\{0,\ldots,l-1\}. \footnotemark
    \end{equation*}
    \footnotetext{These are all the universal $R$-matrices of the Hopf algebra $H_1\cong\mathbb{C}[\mathbb{Z}_l]$ given in \cite[Theorem~3]{wakui1998universal}.} A $G$-integral $\Lambda=(\Lambda_{\alpha^p})_{p=0,\ldots,k-1}$ of $\mathcal{H}$ and an integral $\lambda$ of $H_1^*$ are
    \begin{equation*}
        \Lambda_{\alpha^p} = \frac{1}{l}\sum_{i=0}^{l-1}g^{ik+p}\in H_{\alpha^p},\quad \lambda = l\delta_1\in H_1^*,
    \end{equation*}
    respectively, where $\delta_1$ is the dual basis of $1\in H_1$.
\end{example}

\begin{example}
    Let us first recall the Kac--Paljutkin algebra $H_8$~\cite{kac1966finite,masuoka1995semisimple}. As an algebra, $H_8$ is generated by $x,y,z$ subject to the relations
    \begin{equation*}
        x^2 = y^2 = 1,\quad z^2 = \frac{1}{2}(1+x+y-xy),\quad xy = yx,\quad zx = yz,\quad zy=xz,
    \end{equation*}
    so that $\{1,x,y,xy,z,xz,yz,xyz\}$ forms a basis. The Hopf algebra structure is given by
    \begin{gather*}
        \Delta(x) = x\otimes x,\quad \Delta(y) = y\otimes y,\quad \Delta(z) = \frac{1}{2}(1\otimes 1 + 1\otimes x + y\otimes 1 - y\otimes x)(z\otimes z),\\
        \varepsilon(x) = \varepsilon(y) = \varepsilon(z) = 1,\quad S(x)=x,\quad S(y)=y,\quad S(z)=z.
    \end{gather*}
    This Hopf algebra is the only 8-dimensional involutory Hopf algebra that is neither commutative nor cocommutative.

    Define a linear map $\mu\colon H_8\to H_8$ as the conjugation by $x\in H_8$, i.e., $\mu(h)\coloneqq xhx^{-1}$ for all $h\in H_8$. In particular, $\mu(x) = x$, $\mu(y) = y$, and $\mu(z) = xyz$. Then $\mu$ turns out to be a Hopf algebra automorphism. Since $\mu^2=\id_{H_8}$, it generates a group $G\coloneqq\{1=\id_{H_8},\mu\}\cong\mathbb{Z}_2$.
    
    We are now able to construct a Hopf $G$-algebra $\mathcal{H}=\{H_1,H_\mu\}$ out of $H_8$. For each $\alpha\in G$, $H_\alpha$ is a copy of $H_8$ as a vector space. Let $\iota_\alpha\colon H_8\to H_\alpha$ be the identity. We then define structure morphisms on $\mathcal{H}$ by
    \begin{description}[labelwidth=5em,labelindent=2em,font=\normalfont]
        \item[product] $m_{\alpha,\beta}(\iota_\alpha(a)\otimes \iota_\beta(b)) \coloneq \iota_{\alpha\beta}(a\alpha(b))$,
        \item[unit] $\eta(1) \coloneq \iota_1(1)$,
        \item[coproduct] $\Delta_\alpha(\iota_\alpha(a)) \coloneq \iota_\alpha(a^{(1)})\otimes \iota_\alpha(a^{(2)})$,
        \item[counit] $\varepsilon_\alpha(\iota_\alpha(a)) \coloneq \varepsilon(a)$,
        \item[antipode] $S_\alpha(\iota_\alpha(a)) \coloneq i_{\alpha^{-1}}(S(\alpha^{-1}(a)))=\iota_{\alpha^{-1}}(\alpha^{-1}(S(a)))$.
    \end{description}
    Here $\alpha,\beta\in G$ and $a,b\in H_8$.  The Hopf $G$-algebra $\mathcal{H}$ admits a nontrivial crossing $\varphi=\{\varphi_\beta^\alpha\}_{\alpha,\beta\in G}$ given by
    \begin{equation*}
        \varphi_\beta^\alpha(\iota_\alpha(a))\coloneqq\iota_{\beta\alpha\beta^{-1}}(\beta(a)),
    \end{equation*}
    and with this crossing, it has a universal $R$-matrix given by
    \begin{equation*}
        R=\frac{1}{2}(\iota_1(1)\otimes\iota_1(1)+\iota_1(x)\otimes\iota_1(1)+\iota_1(1)\otimes\iota_1(y)-\iota_1(x)\otimes\iota_1(y)).
    \end{equation*}
    A $G$-integral $\Lambda=(\Lambda_\alpha)_{\alpha\in G}$ of $\mathcal{H}$ and an integral $\lambda$ of $H_1^*$ are
    \begin{gather*}
        \Lambda_\alpha = \frac{1}{8}(\iota_\alpha(1)+\iota_\alpha(x)+\iota_\alpha(y)+\iota_\alpha(xy)+\iota_\alpha(z)+\iota_\alpha(xz)+\iota_\alpha(yz)+\iota_\alpha(xyz)),\\
        \lambda = 8\delta_1,
    \end{gather*}
    respectively.
\end{example}

\section{Colored Kirby diagrams}\label[section]{sec:coloredkirby}

In this section, we discuss geometric concepts used in the construction of the invariant. The main object we use is what we call a $G$-colored Kirby diagram of a flat $G$-connection, which is essentially a Kirby diagram of the base 4-manifold whose dotted components are assigned elements of $G$. To obtain this, we calculate a presentation of the fundamental group of the base 4-manifold from its Kirby diagram and determine a $G$-coloring using the flat $G$-connection. (By a 4-manifold we mean a 4-dimensional smooth oriented connected closed manifold.)

\subsection{Kirby diagrams}
We first quickly review Kirby diagrams of 4-manifolds. We refer the reader to \cite{kirby1989topology,gompf19994,akbulut20164} for more details.

Let $D^k$ denote the $k$-dimensional closed ball. For a (4-dimensional) $k$-handle $D^k\times D^{4-k}\;(k=0,\ldots,4)$, we call $\partial D^k\times D^{4-k}$ the \emph{attaching region} and $\partial D^k\times 0\cong S^{k-1}$ the \emph{attaching sphere}.

Let $M$ be a 4-manifold. Since $M$ is connected, $M$ admits a handle decomposition that has exactly one $0$-handle. Indeed, if a decomposition initially contains multiple $0$-handles, some $1$-handles must connect distinct $0$-handles. In that case, the connected $0$-handles together with the $1$-handle between them are regarded as a single $0$-handle. We thus choose such a handle decomposition of $M$ and identify the boundary of the unique $0$-handle with $\partial D^4=\mathbb{R}^3\cup\{\infty\}$. A $1$-handle is attached to $\partial D^4$ by embedding its attaching region $\partial D^1\times D^3=D^3\coprod D^3$, and this attachment can be described as the removal of a neighborhood of some properly embedded $2$-disk from the $0$-handle~\cite[Section~5.4]{gompf19994}. We therefore draw the boundary circle of the disk on $\partial D^4$ to specify the attachment of the $1$-handle. We also draw a \emph{dot} on the circle to distinguish it from subsequent $2$-handle attachments. After attaching all the $1$-handles, a $2$-handle is attached to the boundary of the resulting $1$-handlebody by embedding its attaching region $\partial D^2\times D^2$. This attachment is determined by (the image of) the attaching sphere $\partial D^2\times 0\cong S^1$ and a framing, which is parameterized by $\pi_1(\mathrm{SO}(2))\cong\mathbb{Z}$ ($0$-framing corresponds to the framing induced by a Seifert surface). We therefore describe this attachment by drawing the image of the attaching sphere as a framed knot (the framing is represented by an integer). When the $2$-handle is attached to a part of a $1$-handle, we draw the knot so that it passes through the dotted circle corresponding to the $1$-handle.

A \emph{Kirby diagram} $L=L_1\sqcup L_2$ is a framed link diagram in $\partial D^4(=\mathbb{R}^3\cup\{\infty\})$ consisting of the set $L_1$ of the dotted components (with $0$-framings) for the $1$-handles and the set $L_2$ of the undotted framed components for the $2$-handles. Thanks to the fact that a closed 4-manifold is determined (up to diffeomorphism) by the $2$-handlebody of a handle decomposition~\cite{laudenbach1972note}, the Kirby diagram $L$ can successfully specify $M$ (up to diffeomorphism). Therefore, for the 3- and 4-handles, we only write their numbers (or even omit them because they will not be used when defining the invariant).

Furthermore, we make use of a \emph{planar Kirby diagram}, which is a Kirby diagram regularly projected onto the plane $\mathbb{R}^2$ with respect to the height function in such a way that:
\begin{itemize}
    \item The dotted components are projected as disjoint circles, with the undotted components passing vertically through them.
    \item The framing of each component is given by the blackboard framing.
\end{itemize}
Some examples of (planar) Kirby diagrams are given in \cref{fig:kirbydiagram,fig:planarS2S1S1}. In the rest of the paper, by a Kirby diagram we always mean a planar Kirby diagram.

\begin{figure}
    \centering
    \includegraphics[scale=1]{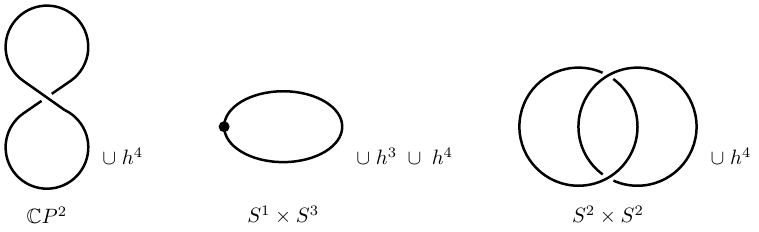}
    \caption{Examples of Kirby diagrams}
    \label{fig:kirbydiagram}
\end{figure}

\begin{figure}
    \centering
    \captionsetup{width=.9\linewidth}
    \includegraphics[scale=1]{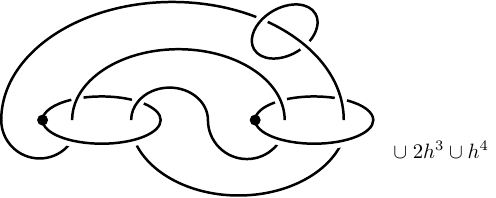}
    \caption{A Kirby diagram of $S^1\times S^1\times S^2$. Let $s_1$ and $s_2$ be generators corresponding to the left and right dotted components, respectively. One can then verify that $\pi_1(S^1\times S^1\times S^2)=\langle s_1,s_2\mid s_1s_2s_1^{-1}s_2^{-1}\rangle\cong \mathbb{Z}\times\mathbb{Z}$.}
    \label{fig:planarS2S1S1}
\end{figure}

\subsubsection{Kirby diagrams and fundamental groups}\label[section]{ssec:kirbyfundgr}

From a Kirby diagram $L=L_1\sqcup L_2$ of $M$, we can obtain a presentation of the fundamental group of $M$ by associating  generators to the dotted component and relations to the undotted components as follows~\cite[Section~2.3.3]{barenz2018dichromatic}:

Each dotted component $\xi_i\in L_1$ corresponds to a generator $s_i$. (Specifically, the generator $s_i$ is represented by a closed curve that passes through the $1$-handle corresponding to $\xi_i$ exactly once and does not pass through any other $1$-handles.) Then, while traversing along each undotted component $\eta_j\in L_2$ in an arbitrary direction from any starting point, we record the generator $s_i$ each time it passes downward through a dotted component $\xi_i$, and its inverse $s_i^{-1}$ each time it passes upward through $\xi_i$. After completing the traversal, we concatenate the stored generators in the order they were recorded, and the resulting word, $r_j$, becomes a relation in the presentation of $\pi_1(M)$. (If $\eta_j$ does not pass through any dotted components, then the corresponding relation is interpreted as $r_j=1$, the trivial one.) Consequently, we get
\begin{equation*}
    \pi_1(M) = \langle s_1,\ldots,s_m\mid r_1,\ldots,r_n\rangle,
\end{equation*}
where $m$ and $n$ are the number of dotted and undotted components of $L$, respectively. \Cref{fig:planarS2S1S1} illustrates an example.

\subsection{Colored Kirby diagrams}\label[subsection]{ssec:coloredkirby}

Here we introduce the notion of a $G$-colored Kirby diagram of a flat $G$-connection.

\begin{definition}\label[definition]{def:connection}
    Let $G$ be a (discrete) group. A \emph{flat $G$-connection} on $M$ is a group homomorphism $\rho\colon\pi_1(M)\to G$. We sometimes refer to the space $M$ as the \emph{base space} of the connection.
    
    Two flat $G$-connections $\rho$ and $\rho'$ on $M$ are said to be \emph{isomorphic} if they differ by the conjugation by an element of $G$, that is, if there is $\beta\in G$ such that $\rho'(s)=\beta\rho(s)\beta^{-1}$ for all $s\in\pi_1(M)$. We also say that two flat $G$-connections $\rho\colon\pi_1(M)\to G$ and $\rho'\colon\pi_1(M')\to G$ on possibly different base spaces are \emph{equivalent} if there is an orientation-preserving diffeomorphism $h\colon M\to M'$ such that $\rho$ and $\rho'h_*$ are isomorphic as flat $G$-connections on $M$, where $h_*\colon\pi_1(M)\to\pi_1(M')$ is the induced group isomorphism.
\end{definition}

Let $\rho\colon\pi_1(M)\to G$ be a flat $G$-connection and $L=L_1\sqcup L_2$ be a Kirby diagram of $M$. We present the fundamental group of $M$ as $\pi_1(M)=\langle s_1,\ldots,s_m\mid r_1,\ldots,r_n\rangle$ as in \cref{ssec:kirbyfundgr} and label each dotted component $\xi_i$ with $\alpha_i\coloneqq\rho(s_i)\in G$. We refer to the label $\alpha_i$ as the \emph{color} of $\xi_i$ and call this assignment a \emph{$G$-coloring} of $L$. A Kirby diagram equipped with a $G$-coloring is called a \emph{$G$-colored Kirby diagram}.

Recall that any two handle decompositions of a manifold can be related by a finite sequence of isotopies of attaching maps and the creation/annihilation of canceling pairs~\cite{cerf1970stratification}. In the context of 4-manifolds, this is achieved by isotopies on levels and five specific moves known as the \emph{Kirby moves}~\cite[Section~5.1]{gompf19994}. In connection with this, $G$-colored Kirby diagrams of two equivalent flat $G$-connections on 4-manifolds are related by the following operations:
\begin{itemize}
    \item isotopies with respect to the height function,
    \item moves involving only undotted components (\cref{fig:moveonlyL2}),
    \item moves involving both dotted and undotted components (\cref{fig:moveL1L2}),
    \item (planar) Kirby moves (\cref{fig:planarkirbymove}),
    \item changing the colors of all the dotted components simultaneously to their conjugates by an element of $G$.
\end{itemize}
The first three operations correspond to projecting Kirby diagrams, and the fourth to the Kirby moves. The last operation, given by conjugation, accounts for differences in $G$-colorings and preserves the equivalence class of $G$-connections.

\begin{figure}
    \centering
    \begin{subcaptionblock}{.48\textwidth}
        \centering
        \includegraphics[scale=1]{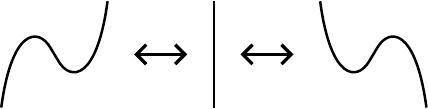}
        \caption*{I-1}
        \label{fig:I-1}
    \end{subcaptionblock}
    \begin{subcaptionblock}{.48\textwidth}
        \centering
        \includegraphics[scale=1]{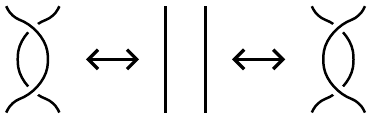}
        \caption*{I-2}
        \label{fig:I-2}
    \end{subcaptionblock}\vspace{0.1in}
    \begin{subcaptionblock}{1\textwidth}
        \centering
        \includegraphics[scale=1]{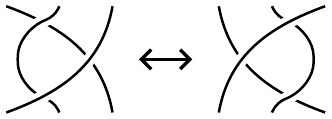}
        \caption*{I-3}
        \label{fig:I-3}
    \end{subcaptionblock}\vspace{0.1in}
    \begin{subcaptionblock}{1\textwidth}
        \centering
        \includegraphics[scale=1]{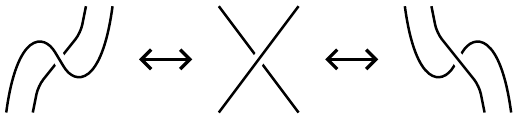}
        \caption*{I-4}
        \label{fig:I-4}
    \end{subcaptionblock}\vspace{0.1in}
    \begin{subcaptionblock}{1\textwidth}
        \centering
        \includegraphics[scale=1]{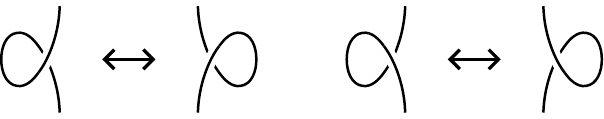}
        \caption*{I-5}
        \label{fig:I-5}
    \end{subcaptionblock}
    \caption{Moves involving only undotted components}
    \label{fig:moveonlyL2}
\end{figure}

\begin{figure}
    \centering
    \begin{subcaptionblock}{.48\textwidth}
        \centering
        \includegraphics[scale=1]{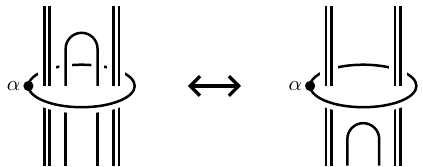}
        \caption*{II-1}
        \label{fig:II-1}
    \end{subcaptionblock}\hfill
    \begin{subcaptionblock}{.48\textwidth}
        \centering
        \includegraphics[scale=1]{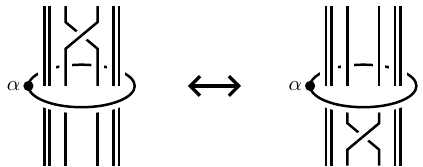}
        \caption*{II-2}
        \label{fig:II-2}
    \end{subcaptionblock}\vspace{0.1in}
    \begin{subcaptionblock}{.48\textwidth}
        \centering
        \includegraphics[scale=1]{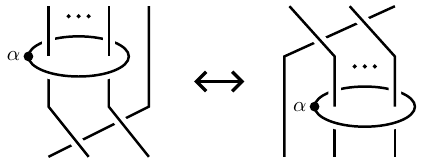}
        \caption*{II-3}
        \label{fig:II-3}
    \end{subcaptionblock}\hfill
    \begin{subcaptionblock}{.48\textwidth}
        \centering
        \includegraphics[scale=1]{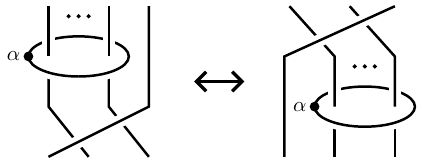}
        \caption*{II-4}
        \label{fig:II-4}
    \end{subcaptionblock}\vspace{0.1in}
    \begin{subcaptionblock}{1\textwidth}
        \centering
        \includegraphics[scale=1]{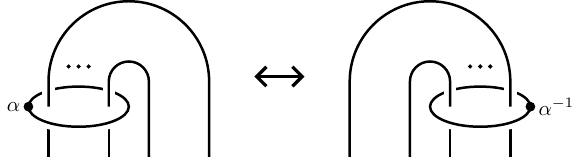}
        \caption*{II-5}
        \label{fig:II-5}
    \end{subcaptionblock}\vspace{0.1in}
    \begin{subcaptionblock}{1\textwidth}
        \centering
        \includegraphics[scale=1]{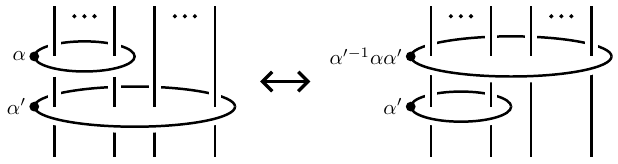}
        \caption*{II-6}
        \label{fig:II-6}
    \end{subcaptionblock}
    \caption{Moves involving both dotted and undotted components}
    \label{fig:moveL1L2}
\end{figure}

\begin{figure}
    \centering
    \captionsetup{width=.9\linewidth}
    \begin{subcaptionblock}{1\textwidth}
        \centering
        \includegraphics[scale=1]{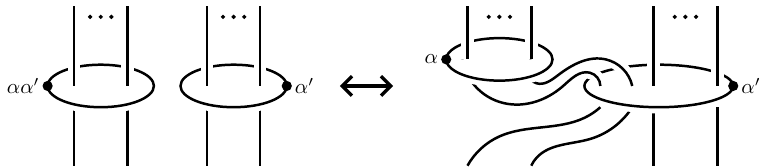}
        \caption*{III-1: $1$-$1$-handle slide}
        \label{fig:III-1}
    \end{subcaptionblock}\vspace{0.1in}
    \begin{subcaptionblock}{.45\textwidth}
        \centering
        \includegraphics[scale=1]{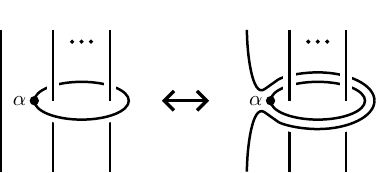}
        \caption*{III-2: $1$-$2$-handle slide}
        \label{fig:III-2}
    \end{subcaptionblock}
    \begin{subcaptionblock}{.45\textwidth}
        \centering
        \includegraphics[scale=1]{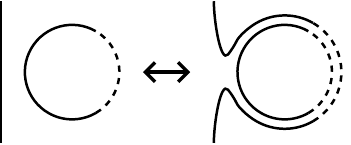}
        \caption*{III-3: $2$-$2$-handle slide}
        \label{fig:III-3}
    \end{subcaptionblock}\vspace{0.1in}
    \begin{subcaptionblock}{.45\textwidth}
        \centering
        \includegraphics[scale=1]{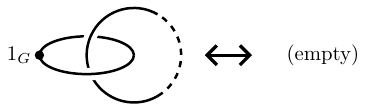}
        \caption*{III-4: $1$-$2$-canceling pair}
        \label{fig:III-4}
    \end{subcaptionblock}
    \begin{subcaptionblock}{.45\textwidth}
        \centering
        \includegraphics[scale=1]{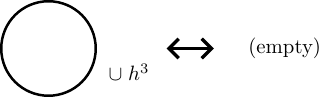}
        \caption*{III-5: $2$-$3$-canceling pair}
        \label{fig:III-5}
    \end{subcaptionblock}
    \caption{(Planar) Kirby moves. In Move~\hyperref[fig:III-3]{III-3}, the strand described by the dashed line can be knotted and linked with other components arbitrarily, and the slide is along the blackboard framing of it. In Move~\hyperref[fig:III-4]{III-4}, the undotted component might be knotted but must not be linked with any other components, except for the dotted component appearing in the figure.}
    \label{fig:planarkirbymove}
\end{figure}

\section{The invariant}\label[section]{sec:invariant}

In this section, we define an invariant of flat connections on 4-manifolds and provide a precise proof of its invariance. We employ a 4-dimensional analog of the Hennings-type construction, and similar techniques can be seen in \cite{bobtcheva2003hkr,chaidez20224,beliakova2022kerler,beliakova2023refined}.

\subsection{Definition}

Let $(\mathcal{H},\varphi,R)$ be a finite type involutory quasitriangular Hopf $G$-algebra. We choose a (two-sided) integral $\Lambda=(\Lambda_\alpha)_{\alpha\in G}$ of $\mathcal{H}$ and a (two-sided) integral $\lambda$ of $H_1^*$ such that $\lambda(\Lambda_1)=1$, $\varepsilon_1(\Lambda_1)=1$. Further, let $\rho\colon\pi_1(M)\to G$ be a flat $G$-connection on a 4-manifold $M$ and $L=L_1\sqcup L_2$ be a Kirby diagram of $M$. As described in \cref{ssec:coloredkirby}, we color the Kirby diagram $L$ and then orient each undotted component arbitrarily.

We first assign elements of $\mathcal{H}$ to the undotted components according to the following procedure. For a non-critical point $q\in\eta$ on an undotted component $\eta\in L_2$, we define $\sigma_q$ to be $0$ if $\eta$ goes downward near $q$, and $1$ if $\eta$ goes upward.
\begin{itemize}
    \item For each dotted component $\xi\in L_1$ colored with $\alpha$, let $q_1,\ldots,q_k$ be the intersections of the disk bounded by $\xi$ with the undotted components passing through it, listed from left to right. We then assign to them the tensor
    \begin{align*}
	\MoveEqLeft (S_\alpha^{\sigma_{q_1}}\otimes\cdots\otimes S_\alpha^{\sigma_{q_k}})(\Delta_\alpha^{(k-1)}(\Lambda_\alpha))\\
        ={}& S_\alpha^{\sigma_{q_1}}(\Lambda_\alpha^{(1)})\otimes\cdots\otimes S_\alpha^{\sigma_{q_k}}(\Lambda_\alpha^{(k)}) \in H_{\alpha^{(-1)^{\sigma_{q_1}}}}\otimes\cdots\otimes H_{\alpha^{(-1)^{\sigma_{q_k}}}}
    \end{align*}
    so that the $i$-th component of the tensor is assigned to $q_i$ (See \cref{fig:contril1l2}). If there are no undotted components passing through $\xi$, then the contribution is interpreted as $\varepsilon_\alpha(\Lambda_\alpha)$, which is $1$ if $H_\alpha\neq0$, and $0$ if $H_\alpha=0$ by \cref{lem:evalginthga} and the assumption $\varepsilon_1(\Lambda_1)=1$.
    \item For each crossing of undotted components (including self-crossing), let $q_1$ and $q_2$ be points on the over and under strands near the crossing, respectively. We then check the sign of the crossing regarding the orientation of the strands (\cref{fig:crossorient}). If the crossing is positive, we assign the universal $R$-matrix $R=a_i\otimes b_i\in H_1\otimes H_1$ so that the first component is assigned to $q_1$ and the second to $q_2$. If the crossing is negative, we assign $R^{-1}=(S_1\otimes\id_{H_1})(R)\in H_1\otimes H_1$ in the same manner as in the positive case.
\end{itemize}

\begin{figure}[t]
    \centering
    \includegraphics[scale=1]{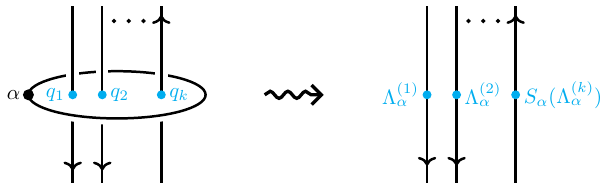}
    \caption{Contribution at a dotted component}
    \label{fig:contril1l2}
\end{figure}

\begin{figure}
    \centering
    \captionsetup{width=.9\linewidth}
    \includegraphics[scale=1]{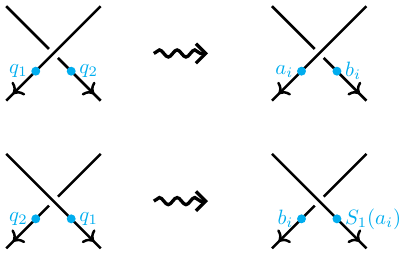}
    \caption{Contribution at a positive crossing (top) and a negative crossing (bottom)}
    \label{fig:crossorient}
\end{figure}

At this stage, the undotted components have received specific elements based on the procedure above (and for convenience, all other parts are interpreted as being assigned $1_{H_1}\in H_1$). From the observation in \cref{ssec:kirbyfundgr}, the product of the assigned elements on each undotted component, multiplied along its orientation from any starting point, belongs to $H_1$. We can therefore evaluate this value by the integral $\lambda\in H_1^*$ (\cref{fig:contraction}). Note that by \cref{lem:cyccointhga}, the result does not depend on the choice of the starting point. We set a bracket $\langle L,\rho\rangle_\mathcal{H}$ to be the product of the evaluations for all the undotted components, and then define the invariant $I_{\mathcal{H}}(M,\rho,L)$ by
\begin{equation*}
    I_{\mathcal{H}}(M,\rho,L) \coloneq \dim(H_1)^{|L_1|-|L_2|} \langle L,\rho\rangle_\mathcal{H},
\end{equation*}
where $|L_1|$ and $|L_2|$ are the number of dotted and undotted components of $L$, respectively. 

\begin{figure}[t]
    \centering
    \includegraphics[scale=1]{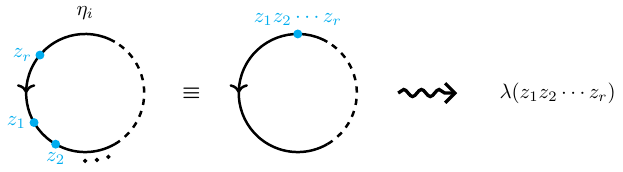}
    \caption{Evaluation of each undotted component}
    \label{fig:contraction}
\end{figure}

\begin{theorem}\label[theorem]{thm:invariance}
    This $I_{\mathcal{H}}(M,\rho,L)$ indeed defines an invariant of flat $G$-connections on 4-manifolds.
\end{theorem}

The proof is given in \cref{ssec:proof}.

\begin{remark}
    When all the dotted components are colored with $1\in G$ (in particular, when $M$ is simply connected or $G$ is trivial), the construction is carried out within the Hopf algebra $H_1$. In this case, the invariant coincides with B{\"a}renz and Barrett's generalized dichromatic invariant~\cite{barenz2018dichromatic} associated with the pivotal functor $\Rep(H_1)\to\Rep(D(H_1))$, where $\Rep(\cdot)$ denotes the category of finite dimensional representations and $D(H_1)$ is the Drinfeld double of $H_1$. The generalized dichromatic invariant is originally defined within a categorical framework, based on a pivotal functor from a spherical fusion category to a ribbon fusion category as algebraic data. However, in the case where the pivotal functor is given as above, the construction is provided using only the language of Hopf algebras in \cite[Section~5]{chaidez20224}. This construction shows the coincidence of our invariant with the generalized dichromatic invariant. Furthermore, by combining this with \cite[Proposition~6.1]{barenz2018dichromatic}, it follows that our invariant also recovers the Crane--Yetter invariant~\cite{crane1993categorical,crane1997state} associated with $\Rep(H_1)$.
\end{remark}

\subsection{Proof of the invariance \texorpdfstring{(\cref{thm:invariance})}{(Theorem 4.2)}}\label[subsection]{ssec:proof}

What we need to show is the following:
\begin{itemize}
    \item Independence of the choice of orientation (\cref{lem:indepori}).
    \item Invariance under the operations listed in \cref{ssec:coloredkirby} (\cref{lem:invkirby}).
\end{itemize}
Note that for the moves in \crefrange{fig:moveonlyL2}{fig:planarkirbymove}, we have to take into account all possible orientations of undotted components.

\begin{lemma}\label[lemma]{lem:indepori}
    The invariant is independent of the choice of orientation of the undotted components.
\end{lemma}

\begin{proof}
    Let $L'$ be an oriented $G$-colored Kirby diagram obtained from $L=L_1\sqcup L_2$ by reversing the orientation of one undotted component, say $\eta\in L_2$. Since $L$ and $L'$ have the same number of each component, it suffices to show that $\langle L,\rho\rangle_\mathcal{H}=\langle L',\rho\rangle_\mathcal{H}$.

    In the diagram $L$, when $\eta$ passes the $i$-th position from the left on the disk bounded by an undotted component colored with $\alpha$, it receives $\Lambda_\alpha^{(i)}$ if it goes downward, and $S_\alpha(\Lambda_\alpha^{(i)})$ if it goes upward. In the diagram $L'$, at this time, $\eta$ receives $S_\alpha(\Lambda_\alpha^{(i)})$ and $\Lambda_\alpha^{(i)}=S_{\alpha^{-1}}(S_\alpha(\Lambda_\alpha^{(i)}))$, respectively. Note that the other components passing through the disk receive the same elements in both diagrams. Next, considering a crossing of undotted components, it easily follows from the assignment rule and $\mathcal{H}$ being involutory that if $\eta$ in $L$ receives an element $z$, then $\eta$ in $L'$ receives $S_1(z)$, and the other component receives the same element in both diagrams. As a result of this observation, if $\eta$ in $L$ is assigned elements $z_1\in H_{\alpha_1},\ldots, z_r\in H_{\alpha_r}$ along its orientation, then the contribution of $\eta$ in $L$ is $\lambda(z_1\cdots z_r)$, and that in $L'$ is also
    \begin{align*}
        \lambda(S_{\alpha_r}(z_r)\cdots S_{\alpha_1}(z_1)) &= \lambda(S_1(z_1\cdots z_r)) && \text{(\Cref{lem:prophga}~(\ref{lem:prophgaa}) and $\alpha_1\cdots\alpha_r=1$)}\\
        &= \lambda(z_1\cdots z_r) && \text{(\Cref{lem:ordinvoequation}~(\ref{lem:ordinvoequatione}))}.
    \end{align*}
    Since the contributions of the other components are the same, we have $\langle L,\rho\rangle_\mathcal{H}=\langle L',\rho\rangle_\mathcal{H}$.
\end{proof}

The above discussion of the orientation choice can also be applied to the subsequent proof of the invariance. Namely, changing the orientation of an undotted component results in applying the antipode to all the assigned elements on it and taking the product in the reverse order of the original orientation. Since $\lambda S_1=\lambda$, the contribution of the component remains the same. Therefore, proving invariance for one orientation case automatically extends to any other orientation cases. Hence, in the following, we will focus on one specific case, without loss of generality.

\begin{lemma}\label[lemma]{lem:invkirby}
    The invariant does not change under the operations described in \cref{ssec:coloredkirby}.
\end{lemma}

\begin{proof}
    Because the operations other than Move~\hyperref[fig:III-5]{III-5} do not alter the normalization term, it suffices to show that the bracket does not change under these operations. Only for Move~\hyperref[fig:III-5]{III-5} do we need to take the coefficient effect into account.\vspace{0.2in}
    
    \noindent\textbf{Isotopies:} This is obvious because isotopies do not change any information regarding crossings.\vspace{0.2in}

    \noindent\textbf{Move~\hyperref[fig:I-1]{I-1}:} In both figures, $1_{H_1}$ is assigned.\vspace{0.2in}

    \noindent\textbf{Move~\hyperref[fig:I-2]{I-2}:} Assume that each component is oriented downward. The element assigned to the leftmost figure is
    \begin{equation*}
        S_1(a_i)a_j\otimes b_ib_j = R^{-1}R = 1_{H_1}\otimes 1_{H_1},
    \end{equation*}
    and that assigned to the rightmost figure is also
    \begin{equation*}
        a_iS_1(a_j)\otimes b_ib_j = RR^{-1} = 1_{H_1}\otimes 1_{H_1}.
    \end{equation*}
    Both are the same as the element assigned to the middle figure.\vspace{0.2in}

    \noindent\textbf{Move~\hyperref[fig:I-3]{I-3}:} Assume that each component is oriented downward. The element assigned to the left figure is
    \begin{equation*}
        a_ja_k\otimes a_ib_k\otimes b_ib_j = R_{23}R_{13}R_{12},
    \end{equation*}
    and that assigned to the right figure is
    \begin{equation*}
        a_ia_j\otimes b_ia_k\otimes b_jb_k = R_{12}R_{13}R_{23}.
    \end{equation*}
    They are equal by \cref{lem:proprmat}~(\ref{lem:proprmatd}) (the quantum Yang--Baxter equation).\vspace{0.2in}

    \noindent\textbf{Move~\hyperref[fig:I-4]{I-4}:} Under this move, the sign of the crossing does not change, and hence neither do the assigned elements.\vspace{0.2in}

    \noindent\textbf{Move~\hyperref[fig:I-5]{I-5}:} Assume that each component is oriented downward. For the first move, the element assigned to the left figure is
    \begin{align*}
        a_ib_i &= S_1(a_i)S_1(b_i) &&\text{(\Cref{lem:proprmat}~(\ref{lem:proprmatc}))}\\
        &= S_1(b_ia_i) &&\text{(\Cref{lem:prophga}~(\ref{lem:prophgaa}))}\\
        &= S_1(u^{-1}) &&\text{(\Cref{lem:drinfeldprop}~(\ref{lem:drinfeldpropa}) and $S_1^2=\id_{H_1}$)}\\
        &= u^{-1} &&\text{(\Cref{lem:invoribbon})},
    \end{align*}
    where $u$ is the Drinfeld element of $(H_1,R)$, and that assigned to the right figure is also
    \begin{equation*}
        b_ia_i = u^{-1}.
    \end{equation*}
    Similarly, for the second move, the element assigned to the left figure is
    \begin{align*}
        b_iS_1(a_i) &= S_1(b_i)a_i &&\text{(\Cref{lem:proprmat}~(\ref{lem:proprmatc}) and $S_1^2=\id_{H_1}$)}\\
        &= u,
    \end{align*}
    and that assigned to the right figure is also
    \begin{align*}
        S_1(a_i)b_i &= S_1(S_1(a_i))S_1(b_i) &&\text{(\Cref{lem:proprmat}~(\ref{lem:proprmatc}))}\\
        &= S_1(b_iS_1(a_i)) &&\text{(\Cref{lem:prophga}~(\ref{lem:prophgaa}))}\\
        &= S_1(u)\\
        &= u &&\text{(\Cref{lem:invoribbon})}.
    \end{align*}
    \vspace{0.1in}

    \noindent\textbf{Move~\hyperref[fig:II-1]{II-1}:} Assume that each component is oriented as shown in \cref{fig:II-1orientation}. The element assigned to the left figure is
    \begin{align*}
        \MoveEqLeft\Lambda_\alpha^{(1)}\otimes\cdots\otimes\Lambda_\alpha^{(i-1)}\otimes S_\alpha(\Lambda_\alpha^{(i)})\Lambda_\alpha^{(i+1)}\otimes\Lambda_\alpha^{(i+2)}\otimes\cdots\otimes\Lambda_\alpha^{(k)}\\
        ={}& \Lambda_\alpha^{(1)}\otimes\ldots\otimes\Lambda_\alpha^{(i-1)}\otimes\varepsilon_\alpha(\Lambda_\alpha^{(i)})1_{H_1}\otimes\Lambda_\alpha^{(i+1)}\otimes\cdots\otimes\Lambda_\alpha^{(k-1)} && \text{(\ref{def:hga9})}\\
        ={}& \Lambda_\alpha^{(1)}\otimes\ldots\otimes\Lambda_\alpha^{(i-1)}\otimes 1_{H_1}\otimes\Lambda_\alpha^{(i)}\otimes\cdots\otimes\Lambda_\alpha^{(k-2)} && \text{(\ref{def:hga4})},
    \end{align*}
    which is the element assigned to the right figure.\vspace{0.2in}

    \begin{figure}
        \centering
        \includegraphics[scale=1]{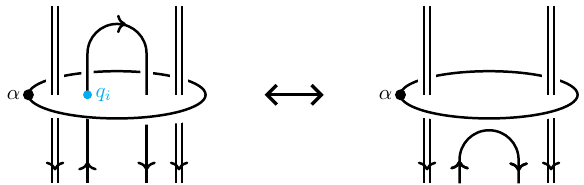}
        \caption{Orientation in Move~\hyperref[fig:II-1]{II-1}}
        \label{fig:II-1orientation}
    \end{figure}

    \noindent\textbf{Move~\hyperref[fig:II-2]{II-2}:} Assume that each component is oriented downward. The element assigned to the left figure is
    \begin{align*}
        \MoveEqLeft\Lambda_\alpha^{(1)}\otimes\cdots\otimes\Lambda_\alpha^{(i-1)}\otimes a_j\Lambda_\alpha^{(i)}\otimes b_j\Lambda_\alpha^{(i+1)}\otimes\Lambda_\alpha^{(i+2)}\otimes\cdots\otimes\Lambda_\alpha^{(k)}\\
        ={}& \Lambda_\alpha^{(1)}\otimes\cdots\otimes\Lambda_\alpha^{(i-1)}\otimes R\Delta_\alpha(\Lambda_\alpha^{(i)})\otimes\Lambda_\alpha^{(i+1)}\otimes\cdots\otimes\Lambda_\alpha^{(k-1)}\\
        ={}& \Lambda_\alpha^{(1)}\otimes\cdots\otimes\Lambda_\alpha^{(i-1)}\otimes\Delta_\alpha^{\mathrm{cop}}(\Lambda_\alpha^{(i)})R\otimes\Lambda_\alpha^{(i+1)}\otimes\cdots\otimes\Lambda_\alpha^{(k-1)} && \text{(\ref{lem:qhga3})}\\
        ={}& \Lambda_\alpha^{(1)}\otimes\cdots\otimes\Lambda_\alpha^{(i-1)}\otimes\Lambda_\alpha^{(i+1)} a_j\otimes\Lambda_\alpha^{(i)} b_j\otimes\Lambda_\alpha^{(i+2)}\otimes\cdots\otimes\Lambda_\alpha^{(k)} &&\text{\Cref{lem:proprmat}~(\ref{lem:proprmata})},
    \end{align*}
    which is the element assigned to the right figure.\vspace{0.2in}

    \noindent\textbf{Move~\hyperref[fig:II-3]{II-3}:} Assume that each of the components passing through the disk is oriented downward, and the other one is oriented upward. The element assigned to the left figure is
    \begin{align*}
        \MoveEqLeft \Lambda_\alpha^{(1)}a_{i_1}\otimes\cdots\otimes\Lambda_\alpha^{(k)}a_{i_k}\otimes b_{i_1}\cdots b_{i_k}\\
        ={}& (\Delta_\alpha^{(k-1)}(\Lambda_\alpha)\otimes 1_{H_1})(a_{i_1}\otimes\cdots\otimes a_{i_n}\otimes b_{i_1}\cdots b_{i_k})\\
        ={}& (\Delta_\alpha^{(k-1)}(\Lambda_\alpha)\otimes 1_{H_1})(a_i^{(1)}\otimes\cdots\otimes a_i^{(k)}\otimes b_{i}) && \text{(\Cref{lem:equatrmat}~(\ref{lem:equatrmata}))}\\
        ={}& (\Delta_\alpha^{(k-1)}(\Lambda_\alpha)\otimes 1_{H_1})(\Delta_1^{(k-1)}(a_i)\otimes b_i)\\
        ={}& (\Delta_1^{(k-1)}(a_i)\otimes b_i)(\Delta_\alpha^{(k-1)}(\Lambda_\alpha)\otimes 1_{H_1})\\
        ={}& (a_i^{(1)}\otimes\cdots\otimes a_i^{(k)}\otimes b_{i})(\Delta_\alpha^{(k-1)}(\Lambda_\alpha)\otimes 1_{H_1})\\
        ={}& a_{i_1}\Lambda_\alpha^{(1)}\otimes a_{i_k}\Lambda_\alpha^{(k)}\otimes b_{i_1}\cdots b_{i_k} && \text{(\Cref{lem:equatrmat}~(\ref{lem:equatrmata}))},
    \end{align*}
    which is the element assigned to the right figure.\vspace{0.2in}
    
    \noindent\textbf{Move~\hyperref[fig:II-4]{II-4}:} Assume that each component is oriented downward. The element assigned to the left figure is
    \begin{align*}
        \MoveEqLeft\Lambda_\alpha^{(1)}b_{i_1}\otimes\cdots\otimes\Lambda_\alpha^{(k)}b_{i_k}\otimes a_{i_k}\cdots a_{i_1}\\
        ={}& (\Delta_\alpha^{(k-1)}(\Lambda_\alpha)\otimes 1_{H_1})(b_{i_1}\otimes\cdots\otimes b_{i_k}\otimes a_{i_k}\cdots a_{i_1})\\
        ={}& (\Delta_\alpha^{(k-1)}(\Lambda_\alpha)\otimes 1_{H_1})(b_i^{(1)}\otimes\cdots\otimes b_i^{(k)}\otimes a_i) && \text{(\Cref{lem:equatrmat}~(\ref{lem:equatrmatb}))}\\
        ={}& (\Delta_\alpha^{(k-1)}(\Lambda_\alpha)\otimes 1_{H_1})(\Delta_1^{(k-1)}(b_i)\otimes a_i)\\
        ={}& (\Delta_1^{(k-1)}(b_i)\otimes a_i)(\Delta_\alpha^{(k-1)}(\Lambda_\alpha)\otimes 1_{H_1})\\
        ={}& (b_i^{(1)}\otimes\cdots\otimes b_i^{(k)}\otimes a_i)(\Delta_\alpha^{(k-1)}(\Lambda_\alpha)\otimes 1_{H_1})\\
        ={}& b_{i_1}\Lambda_\alpha^{(1)}\otimes\cdots\otimes b_{i_k}\Lambda_\alpha^{(k)}\otimes a_{i_k}\cdots a_{i_1} && \text{(\Cref{lem:equatrmat}~(\ref{lem:equatrmatb}))},
    \end{align*}
    which is the element assigned to the right figure.\vspace{0.2in}

    \noindent\textbf{Move~\hyperref[fig:II-5]{II-5}:} Assume that each component is oriented as shown in \cref{fig:II-5orientation}. The element assigned to the right figure is
    \begin{align*}
        \MoveEqLeft S_{\alpha^{-1}}(\Lambda_{\alpha^{-1}}^{(k)})\otimes\cdots\otimes S_{\alpha^{-1}}(\Lambda_{\alpha^{-1}}^{(1)})\\
        ={}& (S_{\alpha^{-1}}\otimes\cdots\otimes S_{\alpha^{-1}})(\Delta_{\alpha^{-1}}^{\mathrm{cop}(k-1)}(\Lambda_{\alpha^{-1}}))\\
        ={}& \Delta_\alpha^{(k-1)}(S_{\alpha^{-1}}(\Lambda_{\alpha^{-1}})) &&\text{(\Cref{lem:prophga}~(\ref{lem:prophgac}))} \\
        ={}& \Delta_\alpha^{(k-1)}(\Lambda_\alpha) &&\text{(\Cref{lem:intantihga})},
    \end{align*}
    which is the element assigned to the left figure.\vspace{0.2in}
    
    \begin{figure}
        \centering
        \includegraphics[scale=1]{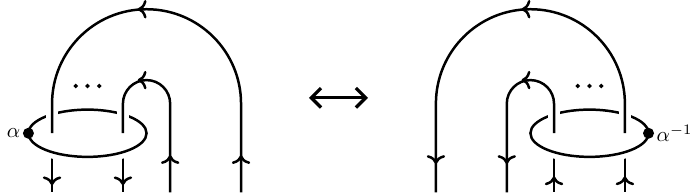}
        \caption{Orientation in Move~\hyperref[fig:II-5]{II-5}}
        \label{fig:II-5orientation}
    \end{figure}

    \hypertarget{invarianceII-6}
    \noindent\textbf{Move~\hyperref[fig:II-6]{II-6}:} Assume that each component is oriented downward. Assume further that, in the left figure, the number of undotted component passing through the large dotted component is $k+l$, of which $k$ also pass through the small dotted component. The element assigned to the left figure is
    \begin{align*}
        \MoveEqLeft\Lambda_\alpha^{(1)}\Lambda_{\alpha'}^{(1)}\otimes\cdots\otimes\Lambda_\alpha^{(k)}\Lambda_{\alpha'}^{(k)}\otimes\Lambda_{\alpha'}^{(k+1)}\otimes\cdots\otimes\Lambda_{\alpha'}^{(k+l)}\\
        ={}& (\Delta_\alpha^{(k-1)}(\Lambda_\alpha)\otimes 1_{H_1}\otimes\cdots\otimes 1_{H_1})\Delta_{\alpha'}^{(k+l-1)}(\Lambda_{\alpha'})\\
        ={}& (\Delta_\alpha^{(k-1)}(\Lambda_\alpha)\otimes 1_{H_1}\otimes\cdots\otimes 1_{H_1})(\Delta_{\alpha'}^{(k-1)}(\Lambda_{\alpha'}^{(1)})\otimes\Delta_{\alpha'}^{(l-1)}(\Lambda_{\alpha'}^{(2)}))\\
        ={}& \Delta_{\alpha\alpha'}^{(k-1)}(\Lambda_\alpha\Lambda_{\alpha'}^{(1)})\otimes\Delta_{\alpha'}^{(l-1)}(\Lambda_{\alpha'}^{(2)})\\
        ={}& \Delta_{\alpha\alpha'}^{(k-1)}(\Lambda_{\alpha\alpha'})\otimes\Delta_{\alpha'}^{(l-1)}(\varepsilon_{\alpha'}(\Lambda_{\alpha'}^{(1)})\Lambda_{\alpha'}^{(2)})\\
        ={}& \Delta_{\alpha\alpha'}^{(k-1)}(\Lambda_{\alpha\alpha'})\otimes\Delta_{\alpha'}^{(l-1)}(\Lambda_{\alpha'}) && \text{(\ref{def:hga4})},
    \end{align*}
    and that assigned to the right figure is
    \begin{align*}
      \MoveEqLeft\Lambda_{\alpha'}^{(1)}\Lambda_{\alpha'^{-1}\alpha\alpha'}^{(1)}\otimes\cdots\otimes\Lambda_{\alpha'}^{(k)}\Lambda_{\alpha'^{-1}\alpha\alpha'}^{(k)}\otimes\Lambda_{\alpha'}^{(k+l)}\otimes\cdots\otimes\Lambda_{\alpha'}^{(l)}\\
      ={}& \Delta_{\alpha'}^{(k+l-1)}(\Lambda_\alpha')(\Delta_{\alpha'^{-1}\alpha\alpha'}^{(k-1)}(\Lambda_{\alpha'^{-1}\alpha\alpha'})\otimes 1_{H_1}\otimes\cdots\otimes 1_{H_1})\\
      ={}& (\Delta_{\alpha'}^{(k-1)}(\Lambda_{\alpha'}^{(1)})\otimes\Delta_{\alpha'}^{(l-1)}(\Lambda_{\alpha'}^{(2)}))\cdot(\Delta_{\alpha'^{-1}\alpha\alpha'}^{(k-1)}(\Lambda_{\alpha'^{-1}\alpha\alpha'})\otimes 1_{H_1}\otimes\cdots\otimes 1_{H_1})\\
      ={}& \Delta_{\alpha\alpha'}^{(k-1)}(\Lambda_{\alpha'}^{(1)}\Lambda_{\alpha'^{-1}\alpha\alpha'})\otimes\Delta_{\alpha'}^{(l-1)}(\Lambda_{\alpha'}^{(2)})\\
      ={}& \Delta_{\alpha\alpha'}^{(k-1)}(\Lambda_{\alpha\alpha'})\otimes\Delta_{\alpha'}^{(l-1)}(\varepsilon_{\alpha'}(\Lambda_{\alpha'}^{(1)})\Lambda_{\alpha'}^{(2)})\\
      ={}& \Delta_{\alpha\alpha'}^{(k-1)}(\Lambda_{\alpha\alpha'})\otimes\Delta_{\alpha'}^{(l-1)}(\Lambda_{\alpha'}) && \text{(\ref{def:hga4})}.
    \end{align*}
    \vspace{0.2in}

    \noindent\textbf{Move~\hyperref[fig:III-1]{III-1}:} Assume that each component is oriented downward. Assume further that, in the left figure, the number of undotted components passing through the left and right dotted components is $k$ and $l$, respectively. The element assigned to the right figure is, by the calculation in Move~\hyperref[fig:II-6]{II-6},
    \begin{equation*}
        \Lambda_\alpha^{(1)}\Lambda_{\alpha'}^{(1)}\otimes\cdots\otimes\Lambda_\alpha^{(k)}\Lambda_{\alpha'}^{(k)}\otimes\Lambda_{\alpha'}^{(k+1)}\otimes\cdots\otimes\Lambda_{\alpha'}^{(k+l)} = \Delta_{\alpha\alpha'}^{(k-1)}(\Lambda_{\alpha\alpha'})\otimes\Delta_{\alpha'}^{(l-1)}(\Lambda_{\alpha'}),
    \end{equation*}
    which is the element assigned to the left figure.\vspace{0.2in}
    
    \noindent\textbf{Move~\hyperref[fig:III-2]{III-2}:} Assume that each component is oriented downward. The element assigned to the right figure is
    \begin{align*}
        \MoveEqLeft b_{i_1}\cdots b_{i_k}a_{j_k}\cdots a_{j_1}\otimes a_{i_1}\Lambda_\alpha^{(1)}b_{j_1}\otimes\cdots\otimes a_{i_k}\Lambda_\alpha^{(k)}b_{j_k}\\
        ={}& b_ia_j\otimes a_i^{(1)}\Lambda_\alpha^{(1)}b_j^{(1)}\otimes\cdots\otimes a_i^{(k)}\Lambda_\alpha^{(k)}b_j^{(k)} && \text{(\Cref{lem:equatrmat}~(\ref{lem:equatrmata}),~(\ref{lem:equatrmatb}))}\\
        ={}& b_ia_j\otimes\Delta_\alpha^{(k-1)}(a_i\Lambda_\alpha b_j)\\
        ={}& \varepsilon_\alpha(a_i)b_ia_j\varepsilon_\alpha(b_j)\otimes\Delta_\alpha^{(k-1)}(\Lambda_\alpha)\\
        ={}& 1_{H_1}\otimes\Delta_\alpha^{(k-1)}(\Lambda_\alpha) && \text{(\ref{def:hga4})},
    \end{align*}
    which is the element assigned to the left figure.\vspace{0.2in}
    
    \noindent\textbf{Move~\hyperref[fig:III-3]{III-3}:} The following observation is essential:

    \begin{lemma}\label[lemma]{lem:parallelcoprod}
        Let $L'$ be an oriented $G$-colored Kirby diagram obtained from $L=L_1\sqcup L_2$ by replacing one undotted component $\eta\in L_2$ with two parallel copies $\eta'$ and $\eta''$, taken along the blackboard framing. Suppose that $\eta$ is assigned an element $z\in H_1$ as a result of the procedure. Then $\eta'$ and $\eta''$ are assigned $z^{(1)}$ and $z^{(2)}$, respectively, where $\Delta_1(z)=z^{(1)}\otimes z^{(2)}$.
    \end{lemma}

    \begin{proof}
        First of all, note that the parallel copies $\eta'$ and $\eta''$ have the same linking with the other components of $L$ as does $\eta$, and in particular they determine the same relation in the fundamental group. Hence, the $G$-coloring is not affected by this replacement.
        
        We then look at each local assignment. At an intersection with the disk bounded by a dotted component with a color $\alpha$, if $\eta$ passes downward through the disc at the $i$-th position, it receives $\Lambda_{\alpha}^{(i)}$. In this case, the copies $\eta'$ and $\eta''$ pass through the $i$-th and $(i+1)$-th positions, respectively, and hence receive $\Lambda_{\alpha}^{(i)(1)}$ and $\Lambda_{\alpha}^{(i)(2)}$. Here we use 
        \begin{equation*}
            \Delta_\alpha(\Lambda_\alpha^{(i)})=\Lambda_{\alpha}^{(i)(1)}\otimes\Lambda_{\alpha}^{(i)(2)}=\Lambda_{\alpha}^{(i)}\otimes\Lambda_{\alpha}^{(i+1)}.
        \end{equation*}
        If $\eta$ goes upward, it receives $S_\alpha(\Lambda_\alpha^{(i)})$. In this case, $\eta'$ and $\eta''$ are at the $(i+1)$-th and $i$-th positions, respectively, and hence receive $S_\alpha(\Lambda_{\alpha}^{(i)})^{(1)}$ and $S_\alpha(\Lambda_{\alpha}^{(i)})^{(2)}$ since
        \begin{equation*}
            (S_\alpha\otimes S_\alpha)(\Delta_\alpha(\Lambda_\alpha^{(i)})) = S_\alpha(\Lambda_{\alpha}^{(i)(1)})\otimes S_\alpha(\Lambda_{\alpha}^{(i)(2)})=S_\alpha(\Lambda_{\alpha}^{(i)})^{(2)}\otimes S_\alpha(\Lambda_{\alpha}^{(i)})^{(1)}.
        \end{equation*}
        In both cases, the assignment to the other components passing through this disk does not change.
        
        At a crossing with an undotted component, say, $\eta$ is the over strand at a positive crossing (\cref{fig:localassign}). Then $\eta$ receives $a_i$ and the other strand receives $b_i$ (where $R=a_i\otimes b_i$). subsequently, $\eta'$ and $\eta''$ receive $a_i$ and $a_j$, respectively, and the other strand receives correspondingly $b_i$ and $b_j$. Then by (\ref{lem:qhga1}), we have
        \begin{equation*}
            a_i\otimes a_j\otimes b_ib_j = R_{13}R_{23} = a_i^{(1)}\otimes a_i^{(2)}\otimes b_i,
        \end{equation*}
        so that the elements assigned to $\eta'$, $\eta''$, and the other strand are $a_i^{(1)}$, $a_i^{(2)}$, and $b_i$, respectively. Other cases are treated similarly by using (\ref{lem:qhga1}) and (\ref{lem:qhga2}).
        
        Overall, whenever $\eta$ receives an element $x$, the copies $\eta'$ and $\eta''$ receive $x^{(1)}$ and $x^{(2)}$, respectively. Since the product and coproduct are compatible, i.e., $(xy)^{(1)}\otimes(xy)^{(2)}=x^{(1)}y^{(1)}\otimes x^{(2)}y^{(2)}$, the lemma follows.
    \end{proof}
    
    \begin{figure}
        \centering
        \begin{subcaptionblock}{.48\textwidth}
            \centering
            \includegraphics[scale=1]{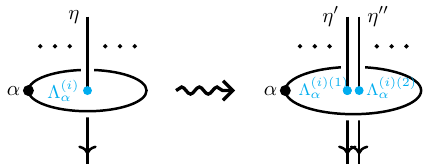}
        \end{subcaptionblock}\hfill
        \begin{subcaptionblock}{.48\textwidth}
            \centering
            \includegraphics[scale=1]{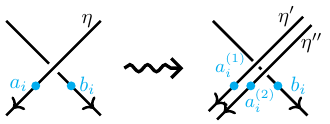}
        \end{subcaptionblock}
        \caption{Local assignment to $\eta$, $\eta'$, and $\eta''$}
        \label{fig:localassign}
    \end{figure}

    Now we proceed to prove the claim. Assume that each component is oriented as shown in \cref{fig:III-3orientation}. Suppose that in the left figure, the right component (the one over which the other slides) is assigned $z\in H_1$ as a result of the procedure. Then the contribution in the left figure is $\lambda(z)1_{H_1}$, and that in the right figure is also
    \begin{equation*}
        \lambda(z^{(1)})z^{(2)} = \lambda(z)1_{H_1}
    \end{equation*}
    by \cref{lem:parallelcoprod} and the definition of the integral.
    \vspace{0.1in}

    \begin{figure}
        \centering
        \includegraphics[scale=1]{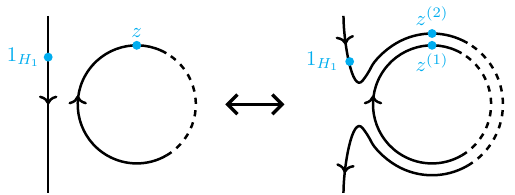}
        \caption{Orientation in Move~\hyperref[fig:III-3]{III-3}}
        \label{fig:III-3orientation}
    \end{figure}
    
    \noindent\textbf{Move~\hyperref[fig:III-4]{III-4}:} We first observe that in a 1-2-canceling pair, the dotted component is colored with $1\in G$ and the undotted component does not intertwine with any other components except for this dotted one. Assume that the undotted component passes downward through the disk. Suppose that the product of the elements resulting from self-intersections of it (in the dashed area), multiplied along the orientation, is $z\in H_1$, so that it is assigned $z\Lambda_1$ as a result. Then the contribution in the left figure is
    \begin{equation*}
        \lambda(z\Lambda_1) = \varepsilon_1(z)\lambda(\Lambda_1) = \varepsilon_1(z).
    \end{equation*}
    Because $z$ is written as the product of the components of the universal $R$-matrix, we can conclude that $\varepsilon_1(z)=1$ by (\ref{def:hga6}) and \cref{lem:proprmat}~(\ref{lem:proprmata}).\vspace{0.2in}
    
    \noindent\textbf{Move~\hyperref[fig:III-5]{III-5}:} This is the only case where we need to consider the normalization coefficient. First note that the $3$-handles (and also $4$-handles) do not contribute to the invariant; their effect may be interpreted as multiplication by the scalar $1\in H_1$. Then the contribution in the left figure is, by \cref{lem:ordinvoequation}~(\ref{lem:ordinvoequationa}) and the assumption $\varepsilon_1(\Lambda_1)=1$,
    \begin{equation*}
        \lambda(1_{H_1}) = \dim(H_1),
    \end{equation*}
    which is canceled out by the normalization term because the number of undotted components is reduced by 1.\vspace{0.2in}

    \noindent\textbf{Coloring change:} Assume that the $G$-coloring changes to its conjugate by $\beta\in G$. We will show $\langle L,\rho\rangle_\mathcal{H}=\langle L,\beta\rho\beta^{-1}\rangle_\mathcal{H}$.

    Suppose that before changing the $G$-coloring, $z_1,\ldots,z_n\in H_1\; (n=|L_2|)$ are assigned to each undotted component of $L$ as a result of the procedure, so that the bracket is given by $\langle L,\rho\rangle_\mathcal{H}=\lambda(z_1)\cdots\lambda(z_n)$. Changing the $G$-coloring from $\rho$ to $\beta\rho\beta^{-1}$ transforms the color $\alpha$ of a dotted component into $\beta\alpha\beta^{-1}$. This leads to the replacement of the associated $G$-integral $\Lambda_\alpha$ with $\Lambda_{\beta\alpha\beta^{-1}}=\varphi_\beta^\alpha(\Lambda_\alpha)$, by \cref{lem:crossinthga}~(\ref{lem:crossinthgaa}). Furthermore, by (\ref{lem:qhga4}), the universal $R$-matrix can be viewed as $R=(\varphi_\beta^1\otimes\varphi_\beta^1)(R)$. We can therefore deduce from \cref{def:crossinghga} and \cref{lem:crosshga} that after changing the $G$-coloring, $z_1,\ldots,z_n$ become $\varphi_\beta^1(z_1),\ldots,\varphi_\beta^1(z_n)$, respectively. Hence, we have
    \begin{align*}
        \langle L,\beta\rho\beta^{-1}\rangle_\mathcal{H} &= \lambda(\varphi_\beta^1(z_1))\cdots\lambda(\varphi_\beta^1(z_n))\\
        &= \lambda(z_1)\cdots\lambda(z_n) &&\text{(\Cref{lem:crossinthga}~(\ref{lem:crossinthgab}))}\\
        &= \langle L,\rho\rangle_\mathcal{H}.
    \end{align*}
    
    This completes the proof of the lemma.
\end{proof}

\begin{notation}
    Since the invariance has been proved, we will use the notation $I_{\mathcal{H}}(M,\rho)$ to denote the invariant for $(M,\rho)$, instead of $I_{\mathcal{H}}(M,\rho,L)$.
\end{notation}

\subsection{Connected sum property and induced 4-manifold invariants}

\subsubsection{Behavior of the invariant under connected sum}

We show that the invariant is multiplicative with respect to connected sum.

Let $\rho\colon\pi_1(M)\to G$ and $\rho'\colon\pi_1(M')\to G$ be two flat $G$-connections. We take a connected sum $M\# M'$ of $M$ and $M'$ so that the orientation of $M\# M'$ is compatible with those of $M$ and $M'$. Since the gluing is performed along $\partial D^4\cong S^3$, which is simply connected, the Seifert--van Kampen theorem ensures that there is a unique homomorphism $\rho\#\rho'\colon\pi_1(M\# M')\to G$ that makes the following diagram commute:
\begin{equation*}
    \begin{tikzcd}
        \pi_1(M)\arrow[rd]\arrow[rrrd, bend left=20, "\rho"]&&&\\
                & \pi_1(M\# M')\arrow[rr, dashed, "\rho\#\rho'"] && G\\
        \pi_1(M')\arrow[ru]\arrow[rrru, bend right=20, "\rho'"']&&&
    \end{tikzcd}
\end{equation*}

\begin{proposition}
    Under the above setting, we have
    \begin{equation*}
        I_\mathcal{H}(M\# M',\rho\#\rho') = I_\mathcal{H}(M,\rho)I_\mathcal{H}(M',\rho').
    \end{equation*}
\end{proposition}

\begin{proof}
    This follows from the fact that a Kirby diagram of $M\# M'$ is obtained by putting the diagrams of $M$ and $M'$ side by side. Note that by the definition of $\rho\#\rho'$, the colors of the dotted components remain the same after taking the connected sum.
\end{proof}

\subsubsection{Induced invariant of 4-manifolds}

As seen so far, our invariant needs to take a flat $G$-connection on a given 4-manifold $M$. In general, there seems to be no standard/interesting choice of a flat $G$-connection on $M$. It is therefore natural to consider invariants that depend only on the manifold itself. One approach to this is to take a sum of the invariants over all possible flat $G$-connections on $M$. However, this approach is not always well-defined; if $M$ admits infinitely many flat $G$-connections, the sum may not converge. Therefore, only when $G$ is finite, do we define an invariant $I_\mathcal{H}$ of 4-manifolds by
\begin{equation*}
    I_\mathcal{H}(M) \coloneqq \sum_{\rho\in\operatorname{Hom}(\pi_1(M),G)} I_\mathcal{H}(M,\rho),
\end{equation*}
where $\operatorname{Hom}(\pi_1(M),G)$ denotes the set of flat $G$-connections on $M$.

\section{Examples}\label[section]{sec:example}

In this section, we give some examples of calculations. We consider the invariants for $\mathbb{C}P^2$, ($\overline{\mathbb{C}P^2}$,) $S^2\times S^2$, $S^1\times S^3$, and $S^1\times S^1\times S^2$ in the setting of the quasitriangular Hopf $\mathbb{Z}_k$-algebra $(\mathcal{H}=\{H_{\alpha^p}\}_{p=0,\ldots,k-1},R_d)$ given in \cref{ex:cyclicgrhga}. Kirby diagrams of these 4-manifolds are presented in \cref{fig:kirbydiagram,fig:planarS2S1S1}.

Before delving into calculations, we will further examine the structures of $(\mathcal{H},R_d)$. Recall that $H_1=\langle 1,g^k,\allowbreak \ldots,g^{(l-1)k}\rangle$ is isomorphic to the group ring $\mathbb{C}[\mathbb{Z}_l]$ as Hopf algebras, and the universal $R$-matrices of $\mathcal{H}$ are given by
\begin{equation*}
    R_d = \frac{1}{l}\sum_{i,j=0}^{l-1}\omega^{-ij}g^{ik}\otimes g^{djk} \in H_1\otimes H_1,, \quad \omega=e^\frac{2\pi\sqrt{-1}}{l}, \quad d\in\{0,\ldots,l-1\}.
\end{equation*}
Following \cite{wakui1998universal}, we introduce the primitive idempotents $E_a\;(a=0,\ldots l-1)$ of $H_1$, which are given by
\begin{equation*}
    E_a \coloneqq \frac{1}{l}\sum_{i=0}^{l-1}\omega^{-ia}g^{ik}.
\end{equation*}
The $E_a$'s satisfy
\begin{gather*}
    S_1(E_a)=E_{l-a},\quad E_aE_b=\delta_{a,b}E_a,\quad \lambda(E_a)=1\quad (a,b=0,\ldots,l-1),\\
    E_0+\cdots+E_{l-1}=1,
\end{gather*}
where $\delta_{a,b}$ is the Kronecker delta and $\lambda=l\delta_1$ is the integral of $H_1^*$. Furthermore, with $E_a$'s, $R_d$ and  $(R_d)_{21}R_d$ can be written as
\begin{equation*}
    R_d=\sum_{i,j=0}^{l-1}\omega^{dij}E_i\otimes E_j,\quad (R_d)_{21}R_d=\sum_{i,j=0}^{l-1}\omega^{2dij}E_i\otimes E_j,
\end{equation*}
where $(R_d)_{21}\coloneqq\tau_{H_1,H_1}(R_d)$.

\subsection{\texorpdfstring{$\mathbb{C}P^2$}{CP2}}
Since the Kirby diagram of $\mathbb{C}P^2$ does not contain dotted components (in particular, $\mathbb{C}P^2$ is simply connected), the calculation is carried out in the Hopf algebra $H_1$. Considering the contribution of only one crossing, we have
\begin{align*}
    I_\mathcal{H}(\mathbb{C}P^2) = I_\mathcal{H}(\mathbb{C}P^2,1) &= \frac{1}{\dim(H_1)}\lambda(a_ib_i)\\
    &= \frac{1}{l}\sum_{i,j=0}^{l-1}\omega^{dij}\lambda(E_iE_j)\\
    &= \frac{1}{l}\sum_{i=0}^{l-1}\omega^{di^2}
\end{align*}
(\cref{fig:cp2calculation}). Furthermore, recall that reversing the orientation of $\mathbb{C}P^2$ gives rise to a non-diffeomorphic manifold, $\overline{\mathbb{C}P^2}$. Its Kirby diagram is the mirror image of that of $\mathbb{C}P^2$, leading to
\begin{equation*}
    I_\mathcal{H}(\overline{\mathbb{C}P^2}) = \frac{1}{\dim(H_1)}\lambda(S_1(b_i)a_i) = \frac{1}{l}\sum_{i=0}^{l-1}\omega^{-di^2}.
\end{equation*}

\begin{figure}
	\centering
	\includegraphics[scale=1]{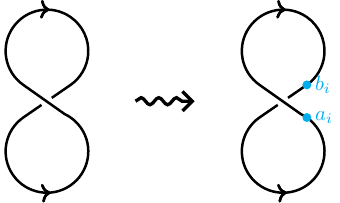}
	\caption{Calculation for $\mathbb{C}P^2$}
	\label{fig:cp2calculation}
\end{figure}

\subsection{\texorpdfstring{$S^2\times S^2$}{S2S2}}
The Kirby diagram of $S^2\times S^2$ also does not contain dotted components (in particular, $S^2\times S^2$ is simply connected). The invariant is
\begin{align*}
    I_\mathcal{H}(S^2\times S^2) = I_\mathcal{H}(S^2\times S^2,1) &= \frac{1}{\dim(H_1)^2}\lambda(b_ia_j)\lambda(a_ib_j)\\
    &= \frac{1}{l^2}(\lambda\otimes\lambda)((R_d)_{21}R_d)\\
    &= \frac{1}{l^2}\sum_{i,j=0}^{l-1}\omega^{2dij}(\lambda\otimes\lambda)(E_i\otimes E_j)\\
    &= \frac{1}{l^2}\sum_{i,j=0}^{l-1}\omega^{2dij}\\
    &= \frac{(l,d)}{l}\left(\frac{3+(-1)^{l/(l,d)}}{2}\right)
\end{align*}
(\cref{fig:s2s2calculation}), where $(l,d)$ is the greatest common divisor of $l$ and $d$.

\begin{figure}
	\centering
	\includegraphics[scale=1]{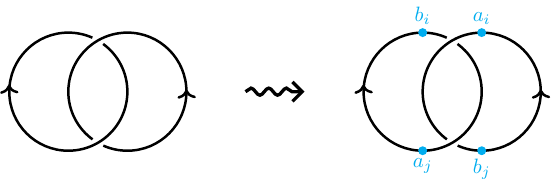}
	\caption{Calculation for $S^2\times S^2$}
	\label{fig:s2s2calculation}
\end{figure}

\subsection{\texorpdfstring{$S^1\times S^3$}{S1S3}}
Since $\pi_1(S^1\times S^3)\cong\mathbb{Z}$, the dotted component of the Kirby diagram of $S^1\times S^3$ can be colored with any element of $\mathbb{Z}_k$. However, since no undotted components are in the diagram, the invariant is simply, by the construction,
\begin{equation*}
    I_\mathcal{H}(S^1\times S^3,\alpha) = \frac{1}{\dim(H_1)^{-1}} = l
\end{equation*}
for every color $\alpha\in\mathbb{Z}_k$. Also, the induced invariant is
\begin{equation*}
    I_\mathcal{H}(S^1\times S^3) = \sum_{\alpha\in\mathbb{Z}_k}I_\mathcal{H}(S^1\times S^3,\alpha)=kl
\end{equation*}

\subsection{\texorpdfstring{$S^1\times S^1\times S^2$}{S2S1S1}}
As in \cref{fig:planarS2S1S1}, we present the fundamental group as $\pi_1(S^1\times S^1\times S^2)=\langle s_1,s_2\mid s_1s_2s_1^{-1}s_2^{-1}\rangle\allowbreak\cong\mathbb{Z}\times\mathbb{Z}$, where $s_1$ corresponds to the left dotted component and $s_2$ to the right. Since $\mathbb{Z}_k$ is abelian, both dotted components can be colored with any element of $\mathbb{Z}_k$. We pick a flat $G$-connection $\rho\colon\pi_1(S^1\times S^1\times S^2)\to \mathbb{Z}_k$ and set $\alpha\coloneqq\rho(s_1)$ and $\beta\coloneqq\rho(s_2)$. Then the invariant is
\begin{align*}
    \MoveEqLeft I_\mathcal{H}(S^1\times S^1\times S^2,\rho)\\
    &= \frac{1}{\dim(H_1)^0}\lambda(b_ia_j\Lambda_\beta^{(2)}S_\alpha(\Lambda_\alpha^{(2)})S_\beta(\Lambda_\beta^{(1)})\Lambda_\alpha^{(1)})\lambda(a_ib_j)\\
    &= \sum_{i,j=0}^{l-1}\omega^{2dij}\lambda(E_i\Lambda_\alpha^{(1)}S_\alpha(\Lambda_\alpha^{(2)})S_\beta(\Lambda_\beta^{(1)})\Lambda_\beta^{(2)})\lambda(E_j) && \text{($\mathcal{H}$ is commutative)}\\
    &= \sum_{i,j=0}^{l-1}\omega^{2dij} && \text{(\ref{def:hga9})}\\
    &= l(l,d)\left(\frac{3+(-1)^{l/(l,d)}}{2}\right)
\end{align*}
(\cref{fig:s1s1s2calculation}). Since the result does not depend on the $G$-connection, the induced invariant is
\begin{equation*}
    I_\mathcal{H}(S^1\times S^1\times S^2) = k^2l(l,d)\left(\frac{3+(-1)^{l/(l,d)}}{2}\right).
\end{equation*}

\begin{figure}
	\centering
	\includegraphics[scale=1]{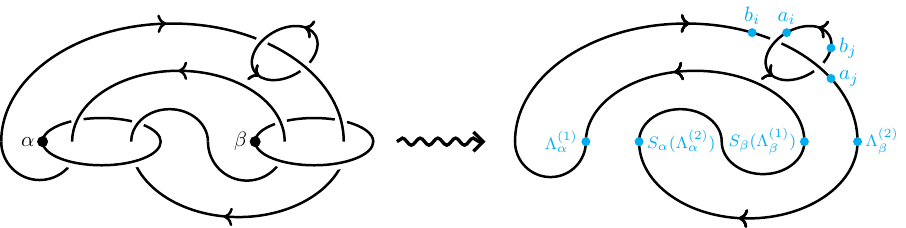}
	\caption{Calculation for $S^1\times S^1\times S^2$}
	\label{fig:s1s1s2calculation}
\end{figure}

\printbibliography[heading=bibintoc]

\end{document}